\documentclass[10pt]{amsart}

\usepackage[colorlinks]{hyperref}
\usepackage{color,graphicx,shortvrb}
\usepackage[latin 1]{inputenc}

\newtheorem{theorem}{Theorem}[section]

\newtheorem{corollary}[theorem]{Corollary}

\newtheorem{lemma}[theorem]{Lemma}

\newtheorem{proposition}[theorem]{Proposition}
\newtheorem{remark}[theorem]{Remark}

\usepackage[active]{srcltx} %SRC Specials for DVI search
\usepackage{stackrel}
\usepackage{amssymb}

\def\J#1#2#3{ \left\{ #1,#2,#3 \right\} }

\def\11{\textbf{1}}

\usepackage{enumerate}
\usepackage{amsmath}
\usepackage{amssymb}
\usepackage{mathabx}

\begin{document}
	
\title{On the strict topology of the multipliers of a JB$^*$-algebra}
	
\author[F.J. Fern\'{a}ndez-Polo]{Francisco J. Fern\'{a}ndez-Polo}
\address[F.J. Fern\'{a}ndez-Polo]{Departamento de An{\'a}lisis Matem{\'a}tico, Facultad de	Ciencias, Universidad de Granada, 18071 Granada, Spain.}
\email{pacopolo@ugr.es}
		
\author[J.J. Garc{\' e}s]{Jorge J. Garc{\' e}s}
\address[J.J. Garc{\' e}s]{Departamento de Matem{\'a}tica Aplicada, ETSIDI, Universidad
	Polt{\' e}cnica de Madrid, Madrid}
\email{j.garces@upm.es}
	
\author[L. Li]{Lei Li}
\address[L. Li]{School of Mathematical Sciences and LPMC, Nankai University, Tianjin 300071, China}
\email{leilee@nankai.edu.cn}

\author[A.M. Peralta]{Antonio M. Peralta}
\address[A.M. Peralta]{Instituto de Matem{\'a}ticas de la Universidad de Granada (IMAG).Departamento de An{\'a}lisis Matem{\'a}tico, Facultad de Ciencias, Universidad de Granada, 18071 Granada, Spain.}
\email{aperalta@ugr.es}

 \dedicatory{Dedicated to the memory and scientific contributions of Professor C.M. Edwards (1941--2021)}
	
	\date{}
	
	\begin{abstract} We introduce the Jordan-strict topology on the multipliers algebra of a JB$^*$-algebra, a notion which was missing despite the fourty years passed after the first studies on Jordan multipliers. In case that a C$^*$-algebra $A$ is regarded as a JB$^*$-algebra, the J-strict topology of $M(A)$ is precisely the well-studied C$^*$-strict topology. We prove that every JB$^*$-algebra $\mathfrak{A}$ is J-strict dense in its multipliers algebra $M(\mathfrak{A})$, and that latter algebra is J-strict complete. We show that continuous surjective Jordan homomorphisms, triple homomorphisms, and orthogonality preserving operators between JB$^*$-algebras admit J-strict continuous extensions to the corresponding type of operators between the multipliers algebras. We characterize J-strict continuous functionals on the multipliers algebra of a JB$^*$-algebra $\mathfrak{A}$, and we establish that the dual of $M(\mathfrak{A})$ with respect to the J-strict topology is isometrically isomorphic to $\mathfrak{A}^*$. We also present a first applications of the J-strict topology of the multipliers algebra, by showing that under the extra hypothesis that $\mathfrak{A}$ and $\mathfrak{B}$ are $\sigma$-unital JB$^*$-algebras, every surjective Jordan $^*$-homomorphism (respectively, triple homomorphism or continuous orthogonality preserving operator) from $\mathfrak{A}$ onto $\mathfrak{B}$ admits an extension to a surjective J-strict continuous Jordan $^*$-homomorphism (respectively, triple homomorphism or continuous orthogonality preserving operator) from $M(\mathfrak{A})$ onto $M(\mathfrak{B})$.
	\end{abstract}

	\maketitle
	\thispagestyle{empty}
	
\section{Introduction}

Multipliers of C$^*$-algebras constitute one of the deepest studied topics in this theory. Multitude of references have been devoted to this object. Let us briefly recall that for each C$^*$-algebra $A$, the \emph{multipliers algebra} $M(A)$ of $A$ can be defined as the idealizer of $A$ in its bidual $A^{**}$ (i.e., the largest
C$^*$-subalgebra of the von Neumann algebra $A^{**}$ containing $A$ as an ideal),  equivalently, $$M(A) =\left\{ x\in A^{**} : x A, A x\subseteq A\right\}$$ (cf. \cite{Busby68,AkPedTom73} and \cite[\S 3.12]{Ped}). The multipliers algebra of a C$^*$-algebra is introduced with the aim of finding an appropriate unital extension within the smallest one obtained by adjoining a unit to $A$, and the largest natural one given by its second dual. Clearly, when $A$ is unital everything collapses and $A = M(A)$.\smallskip

In the wider setting of JB- and JB$^*$-algebras, the multipliers algebra was introduced and studied by C. M. Edwards in \cite{Ed80}. It is worth to refresh the basic terminology. A real or complex \emph{Jordan algebra} is a non-necessarily associative algebra $\mathfrak{A}$ over $\mathbb{R}$ or $\mathbb{C}$ whose product (denoted by $\circ$) is commutative and satisfies the so-called \emph{Jordan identity}: $$( x \circ y ) \circ x^2 = x\circ ( y\circ x^2 ) \hbox{ for all } x,y\in \mathfrak{A}.$$ Jordan algebras are \emph{power 
associative}, that is, each subalgebra generated by a single element $a$ is associative, equivalently, by setting $a^0 = \textbf{1}$, and $a^{n+1} = 
a\circ a^n,$ we have $a^n \circ a^m = a^{n+m}$ for all $n,m\in \mathbb{N}\cup \{0\}$. (cf. \cite[Lemma 2.4.5]{HOS} or \cite[Corollary 1.4]{AlfsenShultz2003}). Given an element $a \in \mathfrak{A}$ the symbol $U_a$ will stand for the linear mapping on $\mathfrak{A}$ defined by $$U_a (b) := 2(a\circ b)\circ a - a^2\circ b.$$ One of the fundamental identities in Jordan algebra theory assures that \begin{equation}\label{eq main identity U maps} U_{U_a(b)} = U_a U_b U_a 
\end{equation} for all $a,b$ in a Jordan algebra (cf. \cite[2.4.18]{HOS} or \cite[$(1.16)$]{AlfsenShultz2003}).\smallskip

A Jordan-Banach algebra $\mathfrak{A}$ is a Jordan algebra equipped with a complete norm satisfying $\|a\circ b\|\leq \|a\| \cdot \|b\|$ for all $a,b\in \mathfrak{A}$. A JB$^*$-algebra is a complex Jordan-Banach algebra $\mathfrak{A}$ equipped with an algebra involution $^*$ satisfying the following Jordan version of the Gelfand-Naimark axiom: $$\|U_a ({a^*}) \|= \|a\|^3, \ (a\in \mathfrak{A}).$$ %, where $U_a(b)  = 2 (a\circ b) \circ a - a^2 \circ b$ for all $b\in \mathfrak{A}$.
A \emph{JB-algebra} is a real Jordan algebra $\mathfrak{J}$ equipped with a complete norm satisfying \begin{equation}\label{eq axioms of JB-algebras} \|a^{2}\|=\|a\|^{2}, \hbox{ and } \|a^{2}\|\leq \|a^{2}+b^{2}\|\ \hbox{ for all } a,b\in \mathfrak{J}.
\end{equation} The key result connecting the notions of JB- and JB$^*$-algebras is due to J. D. M. Wright and affirms that every JB-algebra $\mathfrak{J}$ corresponds uniquely to the self-adjoint part $\mathfrak{A}_{sa}=\{x\in \mathfrak{A} : x^* =x\}$ of a JB$^*$-algebra $\mathfrak{A}$ \cite{Wright77}. A JBW$^*$-algebra (respectively, a JBW-algebra) is a JB$^*$-algebra (respectively, a JB-algebra) which is also a dual Banach space. Every JBW$^*$-algebra (respectively, each JBW-algebra) contains a unit element (see \cite[\S 4]{HOS} or \cite{AlfsenShultz2003}). It is worth to note that JBW-algebras are precisely the self-adjoint parts of JBW$^*$-algebras (see \cite[Theorems 3.2 and 3.4]{Ed80JBW} or \cite[Corollary 2.12]{MarPe2000}).\smallskip 

The class of JB$^*$-algebras attracts more interest by observing that it contains all C$^*$-algebras when equipped with their original norm and involution and the natural Jordan product given by $a \circ b := \frac12 (a b + ba)$. Although the self-adjoint part of a C$^*$-algebra $A$ is not, in general, a subalgebra of $A$, it is always a JB-algebra for the natural Jordan product. The JB$^*$-subalgebras of C$^*$-algebras are called \emph{JC$^*$-algebras}. The statement affirming that the class of JB$^*$-algebras is strictly wider than that of JC$^*$-algebras is confirmed by the existence of exceptional JB$^*$-algebras which cannot be embedded as Jordan $^*$-subalgebras of some $B(H)$ (cf.  \cite[Corollary 2.8.5]{HOS}).%, \cite[Example 3.1.56]{Cabrera-Rodriguez-vol1}).\smallskip
We have already employed in the previous paragraphs the fact that the bidual of each C$^*$-algebra can be equipped with a natural (Arens) product making it a von Neumann algebra with the aim of producing a big unitization. An analogous result is also valid for JB$^*$-algebras, the bidual of each JB$^*$-algebra is a JBW$^*$-algebra, and hence unital (see \cite[Corollary 2.50]{AlfsenShultz2003} or \cite[Theorem 4.4.3]{HOS}). We refer to \cite{AlfsenShultz2003, HOS, Wright77} as the basic references on JB- and JB$^*$-algebras.\smallskip

A \emph{Jordan} homomorphism between Jordan algebras $\mathfrak{A}$ and $\mathfrak{B}$ is a linear mapping $\Phi: \mathfrak{A}\to \mathfrak{B}$ satisfying $\Phi (x\circ y) = \Phi (x) \circ \Phi(y)$, for all $x,y\in \mathfrak{A}$. If additionally $\mathfrak{A}$ and $\mathfrak{B}$ are JB$^*$-algebras and $\Phi (a^*) = \Phi(a)^*$ for all $a\in \mathfrak{A}$ we say that $\Phi$ is a Jordan $^*$-homomorphism. Every Jordan $^*$-homomorphism between JB$^*$-algebras is continuous and non-expansive \cite[Lemma 1]{BaDanHorn}.\smallskip

We can now present the definition of the multipliers algebra given by Edwards in \cite{Ed80}. Let $\mathfrak{A}$ be a JB$^*$-algebra. The \emph{(Jordan) multipliers algebra} of $\mathfrak{A}$ is defined as $$M(\mathfrak{A}):=\{x\in \mathfrak{A}^{**}: x\circ  \mathfrak{A} \subseteq  \mathfrak{A}\}.$$ The set of quasi-multipliers of $\mathfrak{A}$, denoted by $QM(\mathfrak{A})$, is the collection of all $a\in \mathfrak{A}^{**}$ such that $U_a (\mathfrak{A}) \subseteq \mathfrak{A}$. Of course, $M(\mathfrak{A}) =\mathfrak{A}$ when $\mathfrak{A}$ is unital. The space $M(\mathfrak{A})$ is a unital JB$^*$-subalgebra of $\mathfrak{A}^{**}$. Moreover, $M(\mathfrak{A})$ is the (Jordan) idealizer of $\mathfrak{A}$ in $\mathfrak{A}^{**}$, that is, the largest JB$^*$-subalgebra of $\mathfrak{A}^{**}$ which contains $\mathfrak{A}$ as a closed Jordan ideal. It is further known that $M(\mathfrak{A})\subseteq QM(\mathfrak{A})$. An alternative characterization of the multipliers of a JB$^*$-algebra $\mathfrak{A},$ in terms of limits of bounded monotone nets, \cite{Ed80}, reads as follows: every bounded monotone increasing net in $\mathfrak{A}_{sa}^{**}$ possesses a least upper bound in $\mathfrak{A}^{**}.$ For each subset $\mathcal{S}\subseteq \mathfrak{A}$ let $\mathcal{S}^m$ (respectively, $\mathcal{S}_m$) stand for the set of least upper (respectively, greatest lower) bounds of monotone increasing (respectively, decreasing) nets in $\mathcal{S}$, and let $\mathcal{S}^{-}$ denote the norm closure of $\mathcal{S}$. If we write $\tilde{\mathfrak{A}}_{sa} = \textbf{1} \mathbb{R}+ \mathfrak{A}_{sa}$, where $\textbf{1}$ in the unit of $\mathfrak{A}^{**}$, we have $$ M(\mathfrak{A}_{sa}) = \left( \tilde{\mathfrak{A}}_{sa} \right)^m \cap \left( \tilde{\mathfrak{A}}_{sa} \right)_m \hbox{ and } QM(\mathfrak{A}_{sa}) = \left(\left( \tilde{\mathfrak{A}}_{sa} \right)^m\right)^{-} \cap \left(\left( \tilde{\mathfrak{A}}_{sa} \right)_m\right)^{-},$$ a result established for C$^*$-algebras by G. K. Pedersen (cf. \cite[Theorem 2.5]{Ped72}).\smallskip

Despite of the multiple applications of the multipliers algebra in the setting of JB$^*$-algebras (see, for example, \cite{BeRo, BuChu92, EssPer2021, GarPe2021JB, GhoHeja19}), there are certain topological aspect of this object which have not been explored and remain hidden. 
Perhaps the most interesting is a detailed study of an appropriate Jordan version of the strict topology in the mutliplier algebra. We must admit that the required tools have not been available until quite recently. 
In the case of a C$^*$-algebra $A$, the \emph{C$^*$-strict topology} of $M(A)$ (denoted by $S(M(A),A)$) is the topology defined by all the semi-norms of the from $\lambda_a (x) := \|a x \|$ and  $\rho_a (x) := \|x a\|$ ($x\in M(A)$), where $a$ runs freely in $A$. 
There is a clear motivation for this topology, for each $x\in M(A)$, the mappings $L_x,R_x: A \to A$ are bounded linear operators in the Banach algebra, $B(A)$ of all bounded linear operators in $A$. Clearly, $\|L_x\| = \|L_x^{**}\| =\|x\|,$ where $L_x^{**} : A^{**}\to A^{**}$ is the bitranspose of $L_x$, and similarly for $R_x$. Under these identifications of $M(A)$ as subalgebra of $B(A)$, the C$^*$-strict topology of $M(A)$ is nothing but the restriction to $M(A)$ of the strong operator topology (SOT) of $B(A)$ to $M(A)$ (see \cite[\S VI.1]{DunSchwaI}).  In the commutative setting, the multipliers algebra of the C$^*$-algebra $C_0(L),$ of all complex continuous functions on a locally compact Hausdorff space $L$ vanishing at infinity is the C$^*$-algebra $C_b(L)$ of all bounded continuous functions on $L$, and the C$^*$-strict topology on $C_0(L)$ is the locally convex topology generated by the seminorms of the form $p_{\varphi} (f) =\|\varphi f\|$ ($f\in C_b(L)$) with $\varphi\in C_0(L)$ \cite[page 108. Problem 21]{Conw1990}.\smallskip

R.C. Busby showed that $A$ is $S(M(A),A)$-dense in $M(A)$ and the latter is $(M(A), S(M(A),A))$-complete. D. C. Taylor added that, for each C$^*$-algebra $A$, the dual of the l.c.s. $(M(A), S(M(A),A))$ with the topology of uniform convergence on $S(M(A),A)$-bounded subsets of $M(A)$ is a Banach space isometrically isomorphic to the dual of $A$ \cite{Tay70} (a conclusion due to R.C. Buck in the setting of commutative C$^*$-algebra \cite[page 119, Exercise 6]{Conw1990}).\smallskip

In this paper we introduce the J-strict topology of the multipliers algebra of a JB$^*$-algebra $\mathfrak{A}$. We define the \emph{J-strict topology} of $M(\mathfrak{A})$ (denoted by $S(M(\mathfrak{A}),\mathfrak{A})$) as the locally convex topology generated by the seminorms $\{ \rho_a : a\in \mathfrak{A}\},$ where $\rho_a(z)=\|z\circ a\|$ ($z\in M(\mathfrak{A})$), that is, the topology on $M(\mathfrak{A})$ induced by the SOT of $B(\mathfrak{A})$ (better said of $B(\mathfrak{A}_{sa})$) restricted to $M(\mathfrak{A})$ when the elements of $M(\mathfrak{A})_{sa} = M(\mathfrak{A}_{sa})$ are identified with the Jordan multiplication operator that they define, that is, $x\in M(\mathfrak{A})_{sa}\leftrightarrow M_x : \mathfrak{A}\to \mathfrak{A},$ $M_x (a) = x\circ a$. When $\mathfrak{A}$ is unital everything trivializes since $M(\mathfrak{A}) = \mathfrak{A}$ and the strict topology is the norm topology. We note that in case that a C$^*$-algebra $A$ is regarded as a JB$^*$-algebra, its multipliers algebra, $M(A)$, is nothing but its Jordan multipliers algebra \cite{Ed80}, and the J-strict topology on $M(A)$ agrees with the C$^*$-strict topology (cf. Lemma~\ref{l the Jordan and the strict topologies coincide on C*-algebras}). \smallskip

In our first results we prove that every JB$^*$-algebra $\mathfrak{A}$ is J-strictly dense in its multipliers algebra $M(\mathfrak{A})$ (cf. Proposition~\ref{p strict density in the multipliers}). By employing a recent study on continuous linear mappings that are triple derivable at orthogonal pairs from \cite{EssPer2021}, we establish that $M(\mathfrak{A})$ is precisely the J-strict completion of $M(\mathfrak{A})$ (see Theorem~\ref{t the strict topology is complete}). In propositions~\ref{p extensions of surjective Jordan homomorphisms} and~\ref{p extensions of surjective triple homomorphisms} and Corollary~\ref{c extension of OP } we discuss the possibility of finding J-strict continuous extensions of surjective continuous Jordan homomorphisms, triple homomorphisms and continuous linear orthogonality preserving operators between JB$^*$-algebras to similar types of maps between their multipliers algebras. It is known that the extensions are not, in general, surjective (cf. \cite[3.12.11]{Ped} and \cite[Proposition 6.8]{LanceBook} for counterexamples appearing in the case of C$^*$-algebras).\smallskip

Section~\ref{sec:3 Taylor topological dual of Multipliers with the strict topology} is completely devoted to the study of those functionals on the multipliers algebra of a JB$^*$-algebra $\mathfrak{A}$ which are continuous with respect to the J-strict topology. The key characterization, which is a Jordan extension of the result by Taylor in \cite[Corollary 2.2]{Tay70}, appears in Proposition~\ref{p dual of M(A) with strict topology} where we establish that a functional $\phi : M(\mathfrak{A})\to \mathbb{C}$ is J-strict continuous if and only if there exist  $a\in \mathfrak{A}$ and $\varphi\in \mathfrak{A}^*$ such that $\phi (x) = \varphi \left( U_a (x) \right)$ for all $x\in M(\mathfrak{A})$. Furthermore, in coherence with what is known for C$^*$-algebras \cite{Tay70}, the dual of $(M(\mathfrak{A}), S(M(\mathfrak{A}),\mathfrak{A}))$, is isometrically isomorphic to $\mathfrak{A}^*$ (see Theorem~\ref{t dual of multipliers isomorphic to the dual of A}).  The tools in this section include a Cohen factorization type theorem for the dual space of a JB$^*$-algebra $\mathfrak{A}$ regarded as a Jordan-Banach module (see Corollary~\ref{c dual factorizes} and \cite{AkkLaa95}). In Proposition~\ref{p necesary condition for strict equicontinuity} we establish a necessary condition for a family of J-strict continuous functionals on the multipliers algebra of a JB$^*$-algebra to be J-strict equicontinuous.\smallskip

As we have commented before, surjective Jordan $^*$-homomorphisms, triple homomorphisms and continuous linear orthogonality preserving operators between JB$^*$-algebras can be extended to similar types of maps between their multipliers algebras, however the extension is not, in general surjective. In the setting of C$^*$-algebras G.K. Pedersen proved that every surjective $^*$-homomorphism between $\sigma$-unital C$^*$-algebras extends to a surjective $^*$-homomorphism between their corresponding multipliers algebras (see \cite[Theorem 10]{PedSAW} or \cite[Proposition 3.12.10]{Ped}). The main goal in section~\ref{sec 4: surjective extensions} is to prove an appropriate version of Pedersen's result for JB$^*$-algebras, the desired conclusions are achieved in Theorem~\ref{t extension of onto Jordan starhom sigma unital}, Proposition \ref{p extension of onto triple hom sigma unital} and Corollary \ref{c extension of OP surjective with strictly positive element}. Some special tools have been developed to prove the results, we highlight among them an intermediate value type theorem for Jordan $^*$-epimorphisms between JB$^*$-algebras, which proves that given a Jordan $^*$-epimorphism between JB$^*$-algebras $\Phi:\mathfrak{A}\to \mathfrak{B}$, and elements $0\leq b \leq d$ in $\mathfrak{B}$ with $\Phi(c)=d$ for some $c\geq 0$, then there exists $a\in \mathfrak{A}$ such that $0\leq a\leq c$ and $b=\Phi(a)$ (cf. Theorem~\ref{t intermediate value theorem for positive jordan star hom}).\smallskip

The paper culminates with two open questions. As we shall observe in the final section, for each C$^*$-algebra $A$ the product of $M(A)$ is separately strict continuous and jointly strict continuous on bounded sets. The validity of these two statements, specially the second one, is an open problem of indubitable interest in the case of JB$^*$-algebras. 

\section{The completion of a JB$^*$-algebra with respect to the strict topology}\label{sec:2 Busby}

Our first goal is to prove the density of every JB$^*$-algebra in its multipliers algebra with respect to the J-strict topology.  For the proof we shall refine some of the arguments in \cite[Theorem 9]{Ed80}. \smallskip

As in the case of C$^*$-algebras, the JB$^*$-subalgebra generated by a single hermitian element in a JB$^*$-algebra is (isometrically) Jordan $^*$-isomorphic to a commutative C$^*$-algebra (see \cite[Theorems 3.2.2 and 3.2.4]{HOS} or \cite[Corollary 1.19]{AlfsenShultz2003}). In particular the continuous functional calculus for hermitian elements makes perfect sense (see \cite[Proposition 1.21]{AlfsenShultz2003}). We shall employ these properties without any explicit mention. \smallskip

We recall for later purposes that for each self-adjoint element $a$ in a JB$^*$-algebra $\mathfrak{A}$ the mapping $U_a$ is a positive operator, that is, it preserves positive elements (cf. \cite[Proposition 3.3.6]{HOS} or \cite[Theorem 1.25]{AlfsenShultz2003}).\smallskip

Let $\mathfrak{A}$ be a JB$^*$-algebra and $J\subseteq \mathfrak{A}$ a (norm closed) subspace of $\mathfrak{A}$. We shall say that $J$ is a (closed Jordan) ideal of $\mathfrak{A}$ if $J\circ \mathfrak{A} \subseteq J$. A Jordan ideal $J$ of $\mathfrak{A}$ is called \emph{essential} in $\mathfrak{A}$ if every non-zero closed Jordan ideal in $\mathfrak{A}$ has non-zero intersection with $J$. Every closed Jordan ideal of a JB$^*$-algebra is self-adjoint (cf. \cite[Theorem 17]{Young78} or \cite[Proposition 3.4.13]{CabRod2014}).  A subspace $I$ of $\mathfrak{A}$ is said to be a \emph{quadratic ideal} of $\mathfrak{A}$ if, for each element $a$ in $\mathfrak{A}$ and
each pair $b_1, b_2$ of elements in $I$, the element $U_{b_1, b_2} (a) = (b_1\circ a)\circ b_2 + (b_2\circ a)\circ b_1 - (b_1\circ b_2)\circ a$ lies in $I$, equivalently, for each element $a$ in $\mathfrak{A}$ and each element $b$
in $I$, $U_b (a)$ lies in $I$.\smallskip

\begin{proposition}\label{p strict density in the multipliers} Every JB$^*$-algebra $\mathfrak{A}$ is J-strict dense in $M(\mathfrak{A}).$
\end{proposition}
 
\begin{proof} The desired conclusion is almost explicit in the proof of \cite[Theorem 9]{Ed80}. Let us fix $b$ in $M(\mathfrak{A})_{sa}$. As in the proof of \cite[Theorem 9$(i)$]{Ed80}, we can get two increasing nets $(c_j)$ and $(-d_k))$ in $\tilde{\mathfrak{A}}_{sa}$ with least upper bounds $b$ and $-b$, respectively, and $c_j\leq b\leq d_k$ for all $j,k$. For each $a\in \mathfrak{A}_{sa},$ an appropriate application of Dini's theorem to the decreasing net $(U_a (d_k-c_j))$ whose greatest lower bound is zero, asserts that $\displaystyle \lim_{j,k} \| U_a(d_k- c_j)\|= 0$. One of the fundamental geometric inequalities in the theory of JB-algebras (see \cite[Lemma 3.5.2$(ii)$]{HOS}) proves that $$\rho_a (b-c_j)^2 = \|a\circ (b-c_j)\|^2 \leq \|U_a (b-c_j)\| \|b-c_j\| \leq  \| U_a(d_k- c_j)\|   \|b-c_j\| \to_{j,k} 0.$$ By linearity we also get $\rho_a (b-c_j)^2 \to 0$ for all $a\in \mathfrak{A}$. This shows that every $b$ in $M(\mathfrak{A})_{sa}$ lies in the J-strict closure of $\mathfrak{A}_{sa}$. The rest is clear from the identity $\mathfrak{A} = \mathfrak{A}_{sa}\oplus i \mathfrak{A}_{sa}$. 
\end{proof}

Our next goal is to study the completeness of the J-strict topology (cf.  \cite{Busby68}). In this case, the literature contains no forerunners in the Jordan setting nor close approaches. The proof will require recently developed tools.  The arguments in the Jordan setting are based on a much elaborated proof. \smallskip

We recall some basic notions on derivations. A {\it derivation} from a Banach algebra $A$ into a Banach
$A$-module $X$ is a linear map $D: A\to X$ satisfying $D(a b) = D(a) b +a  D(b),$ ($a\in A$). A {\it Jordan derivation} from $A$ into $X$ is a linear map $D$ satisfying $D(a^2) = a D(a) + D(a)
a,$ ($a\in A$), or equivalently, $D(a\circ b)=a\circ D(b)+ D(a)\circ b$ ($a,b\in A$), where $a\circ b = \frac{a b+b a}{2},$ whenever $a,b\in A$, or one of $a,b$ is in $A$ and the other is in $X$. Let $x$ be an element of $X$,
the mapping $\hbox{adj}_{x} :A \to X$, $a\mapsto \hbox{adj}_{x} (a) := x a -  a x$, is an example of a derivation from $A$ into $X$. A derivation $D: A\to X$ is said to be \emph{inner} when it can be written in the form $D = \hbox{adj}_{x}$ for some $x\in X$.\smallskip

Following \cite{LiPan} we shall say that a linear mapping $G$ from a unital Banach algebra $A$ to a (unital) Banach $A$-bimodule $X$ is a \emph{generalized derivation} if it satisfies $$G (ab) = G(a) b + a G(b) - a G(1) b,$$ for all $a,b$ in $A$. The definition for non-necessarily unital Banach algebras and modules was studied in \cite[\S 4]{AlBreExVill09}. Concretely, a \emph{generalized derivation} from a Banach algebra $A$ to a Banach $A$-bimodule $X$ is a linear mapping $D: A \to X$ for which there exists $\xi\in  X^{**}$ satisfying $$D(ab) = D(a)  b + a  D(b) - a \xi b \hbox{ ($a, b \in A$).}$$ 

Clearly, every derivation is a generalized derivation, while there exist generalized derivations which are not derivations. For example, given $x$ in $X$, the Jordan multiplication operator $G_{x} :A \to X$, $x\mapsto G_{x} (a):= a x + x a$, is a generalized derivation from $A$ into $X$. If $A$ is a C$^*$-algebra $A$ and $a$ is an element in $A$ with $a^* \neq -a$, the mapping $G_a : A \to A$ is a generalized derivation which is not a derivation (cf. \cite[comments after Lemma 3]{BurFerGarPe2014}).\smallskip

Suppose now that $X$ is a Jordan-Banach module over a Jordan-Banach algebra $\mathcal{J}$ (see \cite{HejNik96, HoPeRu, PeRu2014} for the detailed definition of Jordan-Banach module). We shall say that a linear mapping $D: \mathcal{J}\to X$ is a \emph{Jordan derivation} if the identity $$D(a\circ b ) = D(a)\circ b + a \circ D(b)$$ holds for all $a,b\in\mathcal{J}$. Given $x\in X$ and $a\in \mathcal{J}$, we set $L(a) (x) =  a\circ x$ and $L(x)(a) = a\circ x$. By a little abuse of notation, we also denote by $L(a)$ the operator on $\mathcal{J}$ defined by $L(a) (b)  = a\circ b$. Let us fix $a\in \mathcal{J}$ and $x\in X$, the mapping $$[L(x), L(a)] = L(x) L(a) -L(a) L(x) : \mathcal{J} \to X, \ b\mapsto [L(x), L(a)] (b),$$ is a Jordan derivation. %A derivation $D: \mathcal{J} \to X$ that can be written in the form $D= \sum_{i=1}^{m} \left(L(x_i)L(a_i)-L(a_i)L(x_i)\right)$, ($x_i\in X, a_i\in \mathcal{J}$) is called \emph{inner}. \smallskip
A linear mapping $G: \mathcal{J}\to X$ will be called a \emph{generalized Jordan derivation} if we can find $\xi\in X^{**}$ satisfying \begin{equation}\label{eq generalized Jordan derivation} G (a\circ b) = G(a)\circ b + a\circ G(b) - U_{a,b} (\xi ),
\end{equation} for every $a,b$ in $\mathcal{J}$ (cf. \cite{AlBreExVill09, BurFerGarPe2014, BurFerPe2013}).\smallskip

Concerning automatic continuity, B. Russo and the last author of this note showed in \cite[Corollary 17]{PeRu2014} that every Jordan derivation from a C$^*$-algebra $A$ into a Banach $A$-bimodule or into a Jordan-Banach $A$-module is continuous. It is further known that every Jordan derivation from a JB$^*$-algebra $\mathfrak{A}$ into $\mathfrak{A}$ or into $\mathfrak{A}^*$ is automatically continuous (cf. \cite[Corollary 2.3]{HejNik96} and also \cite[Corollary 10]{PeRu2014}). To complete the state-of-the-art, we note that F.B. Jamjoom, A. Siddiqui and the last author of this paper proved in  that \cite[Proposition 2.1]{JamPeSidd2015} that every generalized Jordan derivation from a JB$^*$-algebra $\mathfrak{A}$ into $\mathfrak{A}$ or into $\mathfrak{A}^{**}$ is automatically continuous.\smallskip

As we shall comment later, Jordan derivations and generalized Jordan derivations do not suffice to prove the completeness of the multipliers algebra with respect to the J-strict topology. For this purpose the appropriate tools are triple derivations and the structure of JB$^*$-triple underlying every JB$^*$-algebra. A \emph{JB$^*$-triple} is a complex Banach space $E$ equipped with a continuous triple product $\{\cdot,\cdot,\cdot\}:E\times E\times E\rightarrow E$ which is linear and symmetric in the outer variables, conjugate
linear in the middle one and satisfies the following conditions:
\begin{enumerate}[(JB$^*$-1)]
	\item (Jordan identity) for $a,b,x,y,z$ in $E$,
	$$\{a,b,\{x,y,z\}\}=\{\{a,b,x\},y,z\}
	-\{x,\{b,a,y\},z\}+\{x,y,\{a,b,z\}\};$$
	\item $L(a,a):E\rightarrow E$ is an hermitian
	(linear) operator with non-negative spectrum, where $L(a,b)(x)=\{a,b,x\}$ with $a,b,x\in
	E$;
	\item $\|\{x,x,x\}\|=\|x\|^3$ for all $x\in E$,
\end{enumerate} see \cite{Ka} for the original reference and holomorphic motivation of the model. We also refer to the monographs \cite{CabRod2014, Chu2012} for the basic background on JB$^*$-triples. To let the reader have and idea of the size of the class of JB$^*$-triples, we shall simply observe that it is strictly wider than the classes of C$^*$- and JB$^*$-algebras. The triple products inducing an structure of JB$^*$-triple on the Banach spaces in the just mentioned categories are given by $\{a,b,c\} := \frac12 (a b^* c + c b^* a)$ and $\{ x,y,z\} = (x\circ y^*) \circ z + (z\circ y^*)\circ x -
(x\circ z)\circ y^*,$ respectively.\smallskip

In the sequel we need to deal with tripotents in JB$^*$-algebras and JB$^*$-triples, a notion that generalizes the concept of partial isometry in C$^*$-algebras. Each element $e$ in a JB$^*$-triple $E$ satisfying  $\{e,e,e\} =e$ is called a \emph{tripotent}. Associated with each tripotent $e\in E$ we find a  \emph{Peirce decomposition} of $E$ in the form $$E= E_{2} (e) \oplus E_{1} (e) \oplus E_0 (e),$$ where for
$j=0,1,2,$ $E_j (e)$ is the $\frac{j}{2}$ eigenspace of the operator $L(e,e)$. 
%The Peirce decomposition satisfies certain rules known as \emph{Peirce arithmetic}: $$\J {E_{i}(e)}{E_{j} (e)}{E_{k} (e)}\subseteq E_{i-j+k} (e),$$ if $i-j+k \in \{ 0,1,2\}$ and is zero otherwise. In addition, $$\J {E_{2} (e)}{E_{0}(e)}{E} = \J {E_{0} (e)}{E_{2}(e)}{E} =0.$$
%For $j\in \{0,1,2\}$ the corresponding \emph{Peirce $j$-projection} of $E$ onto $E_j (e)$ is denoted by $P_{j} (e)$. 
It is worth to note that the Peirce 2-subspace $E_2 (e)$ is a JB$^*$-algebra with respect to the Jordan product and involution given by $$\hbox{$x\circ_e y := \J xey$ and $x^{*_e} := \J	exe$,}$$ respectively (cf. \cite[\S 4.2.2, Fact 4.2.14 and Corollary 4.2.30]{CabRod2014}). A tripotent $u$ in $E$ is called unitary if $E_2(u) = E$. There is a one-to-one identification among JB$^*$-triples containing a unitary element and unital JB$^*$-algebras.\smallskip

Each unitary $u$ in a unital JB$^*$-algebra $\mathfrak{A}$ induces a new (Jordan) product and an involution defined by $x\circ_u y :=U_{x,y}(u^*)=\{x,u,y\},$ and $x^{*_u} :=U_u(x^*)=\{u,x,u\}$, respectively, in which $u$ acts as the unit element. This new unital JB$^*$-algebra $M(u)= (M,\circ_u,*_u)$ is called the \emph{$u$-isotope} of $M$. In general, and contrary to what happens in the case of C$^*$-algebras, two different tripotents can give rise to completely different JB$^*$-algebras, more concretely, there exist examples of unital JB$^*$-algebras $\mathfrak{A}$ admitting two unitaries $u_1$ and $u_2$ for which we cannot find a surjective linear isometry on $\mathfrak{A}$ mapping $u_1$ to $u_2$ --recall that every Jordan $^*$-isomorphism between JB$^*$-algebras is isometric-- (cf. \cite[Example 5.7]{BraKaUp78}). Although different unitaries in a unital JB$^*$-algebra $\mathfrak{A}$ can produce very different JB$^*$-algebras with the same underlying Banach space, by a celebrated theorem of W. Kaup (see \cite[Proposition 5.5]{Ka}), the triple product is essentially unique in the sense that it satisfies \begin{equation}\label{eq uniqueness of the triple product} \begin{aligned} \J xyz & = (x\circ_{u_1} y^{*_{u_1}}) \circ_{u_1} z + (z\circ_{u_1} y^{*_{u_1}})\circ_{u_1} x -
		(x\circ_{u_1} z)\circ_{u_1} y^{*_{u_1}}\\
		&= (x\circ_{u_2} y^{*_{u_2}}) \circ_{u_2} z + (z\circ_{u_2} y^{*_{u_2}})\circ_{u_2} x -
		(x\circ_{u_2} z)\circ_{u_2} y^{*_{u_2}},
	\end{aligned}
\end{equation} for all $x,y,z\in \mathfrak{A}$ and every couple of unitaries $u_1,u_2$ in $\mathfrak{A}$.\smallskip

A subspace $B$ of a JB$^*$-triple $E$ is a JB$^*$-subtriple of $E$ if $\{B,B,B\}\subseteq B$. 
A (closed) \emph{triple ideal} or simply an \emph{ideal} of $E$ is a (norm closed) subspace $I\subseteq E$ satisfying $\{E,E,I\}+\{E,I,E\}\subseteq I$, equivalently, $\{E,E,I\}\subseteq I$ or $\{E,I,E\}\subseteq I$ or $\{E,I,I\}\subseteq I$ (see \cite[Proposition 1.3]{BuChu92}).  Since every closed Jordan ideal $I$ of a JB$^*$-algebra $\mathfrak{A}$ is self-adjoint, it is also a triple ideal. In a JB$^*$-algebra closed Jordan ideals and closed triple ideals coincide.\smallskip

A JB$^*$-subtriple $I$ of $E$ is called an \emph{inner ideal} of $E$ if $\{I,E,I\}\subseteq I$. A subspace $I$ of a C$^*$-algebra $A$ is an \emph{inner ideal} if $I A I \subseteq I$. Every hereditary $\sigma$-unital C$^*$-subalgebra of a C$^*$-algebra is an inner ideal. A complete study on inner ideals of JB$^*$-triples is available in \cite{EdRu92} and the references therein.\smallskip

By local Gelfand theory, the JB$^*$-subtriple $E_a$, generated by a single element $a$ in a JB$^*$-triple $E$ identifies with a commutative C$^*$-algebra admiting $a$ as positive generator (cf. \cite[Corollary 1.15]{Ka}). Consequently, every element in a JB$^*$-triple admits a cubic root and a $(2n-1)$th-root ($n\in \mathbb{N}$) belonging to the JB$^*$-subtriple that it generates. The sequence $(a^{[\frac{1}{2n-1}]})$ of all  $(2n-1)$th-roots of $a$ converges in the weak$^*$ topology  of $E^{**}$ to a tripotent in $E^{**},$ denoted by $r_{{E^{**}}}(a)$, and called the \emph{range tripotent} of $a$. The tripotent $r_{{E^{**}}}(a)$ is the smallest tripotent $e\in E^{**}$ satisfying that $a$ is positive in the JBW$^*$-algebra $E^{**}_{2} (e)$ (compare \cite[Lemma 3.3]{EdRu88}). %It is also known that, if $\|a\|= 1,$ the sequence $(a^{[2n -1]}),$ of all odd-powers of $a$, converges in the weak$^*$- and strong$^*$ topology of $E^{**}$ to a tripotent (called the
%\hyphenation{support}\emph{support} \emph{tripotent} of $a$, $u(a)$ in $E^{**}$, which satisfies $ u(a) \leq a \leq r_{{E^{**}}}(a)$ in $E^{**}_2 (r_{{E^{**}}}(a))$ (compare \cite[Lemma 3.3]{EdRu88}; beware that in \cite{EdRu96}, $r(x)$ is called the support tripotent of $x$).
 In case that $a$ is a positive element in a JB$^*$-algebra $M$, the range tripotent of $a$ in $M^{**}$ is a projection, called the \emph{range projection} of $a$ in $M^{**}.$ \smallskip

For each element $a$ in a JB$^*$-triple $E$, we shall denote by $E(a)$ the norm closure of $\J aEa = Q(a) (E)$ in $E$. It is known that $E(a)$ is precisely the norm-closed inner ideal of $E$ generated by
$a$. Clearly, $E_a\subset E(a)$ \cite{BunChuZal2000}.  It is also proved in \cite{BunChuZal2000} that $E(a)$ is a JB$^*$-subalgebra of the JBW$^*$-algebra $E(a)^{**} = \overline{E(a)}^{w^*} = E^{**}_{2} (r_{{E^{**}}}(a))$ and contains $a$ as a positive element, where $r_{{E^{**}}}(a)$  is the range tripotent of $x$ in $E^{**}$ (cf. \cite[Proposition 2.1]{BunChuZal2000}).\label{inner ideal}
\smallskip

A triple homomorphism between JB$^*$-triples $E$ and $F$ is a linear mapping $T : E\to F$ satisfying $$T\{a,b,c\} = \{T(a),T(b),T(c)\}, \hbox{ for all $a,b,c\in E$.}$$ A result in the folklore of the theory asserts that every unital triple homomorphism between unital JB$^*$-algebras is a Jordan $^*$-homomorphism (i.e. it preserves Jordan products and involution).\smallskip

A JBW$^*$-triple is a JB$^*$-triple which is also a dual Banach space. JBW$^*$-triples admit a unique isometric predual and their triple product is separately weak$^*$ continuous \cite{BaTi}. Examples of JBW$^*$-triples can be given by just taking the bidual of every JB$^*$-triple \cite{Di86}. Obviously, the class of JBW$^*$-triples includes all von Neumann algebras and JBW$^*$-algebras.\smallskip

Elements $a,b$ in a JB$^*$-triple are called \emph{orthogonal} (denoted $a\perp b$) if $L(a,b)=0$ (see \cite[Lemma 1]{BurFerGarMarPe2008} for other equivalent reformulations). We say that a mapping between JB$^*$-triples is \emph{orthogonality preserving} if it maps orthogonal elements to orthogonal elements.\smallskip

It is perhaps worth to observe that the J-strict topology on the multipliers algebra of a JB$^*$-algebra $\mathfrak{A}$ is a Hausdorff topology. It suffices to show that for $x = h +i k\in M(\mathfrak{A})$ ($h,k\in M(\mathfrak{A})_{sa}$) the condition $x \circ \mathfrak{A} =\{0\}$ implies $x=0$. There are several arguments to obtain it. The first one follows from the separate weak$^*$-continuity of the Jordan product of $\mathfrak{A}^{**}$ and the weak$^*$-density of $\mathfrak{A}$ in its bidual which combined with our assumptions give $h^2 +k^2 =x\circ x^* =0$ and thus $h=k=x=0$. Alternatively, $x\circ \mathfrak{A}=\{0\}$ is equivalent to $h\circ \mathfrak{A}=\{0\} =k\circ \mathfrak{A}=\{0\}$. It is not hard to check that, under these conditions, $h,k\perp \mathfrak{A}$ (see \cite[Lemma 4.1]{BurFerGarPe09}), and thus $x$ lies in the quadratic annihilator of $\mathfrak{A}$ in $M(\mathfrak{A})$, that is, in the set $\mathfrak{A}^{\perp_q} = \{y\in M(\mathfrak{A}) : U_y (\mathfrak{A}) =\{0\} \}$. Since $\mathfrak{A}^{\perp_q}$ is a closed quadratic ideal in $M(\mathfrak{A})$ whose intersection with $\mathfrak{A}$ is clearly zero, and the intersection of $\mathfrak{A}$ with each non-zero quadratic ideal of $M(\mathfrak{A})$ is non-zero (cf. \cite[Proposition 5]{Ed80}), it necessarily follows that $x=0$.  \smallskip

A linear mapping $\delta$ on a JB$^*$-triple $E$ is called a \emph{triple derivation} if it satisfies $$\delta\{a,b,c\} = \{\delta(a), b, c\} + \{a,\delta(b), c\} + \{a,b,\delta(c)\},$$ for all $a,b,c\in E$. Following \cite{EssPer2021} we shall say that a linear mapping $T:E\to E$ is \emph{triple derivable at orthogonal pairs} if 
$$0=\{T(a),b,c\}+\{a,T(b),c\}+\{a,b,T(c)\}$$ for those $a,b,c\in E$ with $a\perp b$. Clearly, every triple derivation enjoys this property.\smallskip 

According to the usual notation, for each linear mapping $T$ on a JB$^*$-algebra $\mathfrak{A}$, we define another linear mapping $T^{\#}:\mathfrak{A}\to \mathfrak{A}$ given by $T^{\#}(a)=T(a^*)^*.$ The mapping $T$ is called symmetric (respectively, anti-symmetric) if $T^{\#}=T$ (respectively, $T^{\#}=-T$). There are fine connections between associative, Jordan and triple derivations, whenever the connections are possible (see \cite[\S 3]{HoPeRu}). Every associative  $^*$-derivation on a C$^*$-algebra $A$ is a triple derivation. On the other direction, if $A$ is unital, for each triple derivation $\delta: A\to A,$ the element $\delta (\textbf{1})$ is skew symmetric in $A$ and the mapping $(\delta - M_{\delta(\textbf{1})}) (x) = \delta(x) - \frac{\delta(\textbf{1}) x + x \delta (\textbf{1}) }{2}$ is a Jordan $^*$-derivation on $A$ (cf. \cite[Lemmata 1 and 2]{HoMarPeRu}). The same conclusions actually hold for Jordan derivations on a unital JB$^*$-algebra. \smallskip

It is shown in \cite[Proposition 4.4]{EssPer2021} that if $T:\mathfrak{A}\to \mathfrak{A}$ is an anti-symmetric linear mapping on a JB$^*$-algebra which is triple derivable at orthogonal pairs, then $T$ is a triple derivation and there exists $z\in M(\mathfrak{A})$ with $z^*=-z$ such that $T(a)=M_{z} (a) =z\circ a.$ This conclusion applies, in particular, to anti-symmetric triple derivations on $\mathfrak{A}.$ Conversely, if $z\in M(\mathfrak{A})$ is skew-symmetric, then $M_z=\frac{1}{2}\delta(z,\textbf{1}) := \frac12 (L(z,\textbf{1})- L(\textbf{1},z))$ (where $\textbf{1}$ stands for the unit of $\mathfrak{A}^{**}$) is an anti-symmetric triple derivation on $\mathfrak{A}$  \cite[Lemma 2]{HoMarPeRu}.\smallskip

\begin{theorem}\label{t the strict topology is complete} $M(\mathfrak{A})$ is the J-strict completion of $\mathfrak{A}$.
\end{theorem}

\begin{proof} We begin by proving that $M(\mathfrak{A})$ is J-strict complete. Let us fix a J-strict Cauchy net  $(z_{\lambda})$ in $M(\mathfrak{A}).$ It follows from the assumptions that, for each $a\in \mathfrak{A},$ the net $(z_{\lambda} \circ a)$ is a norm-Cauchy net in $\mathfrak{A}$ and hence it converges in norm. Let as write $\displaystyle T(a)=\lim_{ \lambda} z_{\lambda} \circ a.$ This allows to define a linear mapping $$T:\mathfrak{A}\to \mathfrak{A}, a\mapsto T(a):=\lim_{ \lambda} z_{\lambda} \circ a.$$

Now, for each $\lambda$, we set $x_{\lambda}=\frac{z_{\lambda}+z_{\lambda}^*}{2}$ and $y_{\lambda}=\frac{z_{\lambda}-z_{\lambda}^*}{2}$. It is not hard to see that $(x_{\lambda})$ and $(y_{\lambda})$ are both J-strict Cauchy nets in $M(\mathfrak{A})$. By applying the arguments in the first paragraph, we define two linear mappings $T_1,T_2:\mathfrak{A}\to \mathfrak{A}$ by $\displaystyle T_1(a):=\displaystyle \lim_{\lambda} x_{\lambda} \circ a $ and $\displaystyle T_2(a):=\lim_{\lambda} y_{\lambda} \circ a $ ($a\in \mathfrak{A}$). Since $y_{\lambda}^*=-y_{\lambda}$ for every $\lambda,$ the (bounded) linear operator $M_{\lambda}:\mathfrak{A}\to \mathfrak{A},$  $M_{\lambda}(a)= y_{\lambda}\circ a$ is an anti-symmetric triple derivation. We claim that $T_2$ also is an anti-symmetric triple derivation. Namely, by the previous observation, the identity $$M_{\lambda}(\{a,b,c\})=\{M_{\lambda}(a),b,c\}+\{a,M_{\lambda}(b),c\}+\{a,b,M_{\lambda}(c)\} $$ holds for all $\lambda$, therefore, by taking norm-limits in the above equality we have 
$$T_2(\{a,b,c\})=\{T_2(a),b,c\}+\{a,T_2(b),c\}+\{a,b,T_2(c)\},$$ where we also applied the continuity of the triple product. Thus $T_2$ is a triple derivation on $\mathfrak{A}$. As a consequence $T_2$ is continuous (\cite{BarFri90} or \cite{PeRu2014}). Furthermore, for each $a\in \mathfrak{A},$ we have 
 $\displaystyle T_2^{\#}(a)=T(a^*)^*=\left(\lim_{\lambda}(y_{\lambda}\circ a^*)\right)^*=-\lim_{\lambda}(y_{\lambda}\circ a)=-T_2(a),$ witnessing that $T_2$ is an anti-symmetric triple derivation on $\mathfrak{A}$. By \cite[Proposition 4.4]{EssPer2021} there exists an  anti-symmetric element $y \in M(\mathfrak{A})$ such that $T_2(a)=y\circ a$. Consequently, $y_{\lambda}$ converges J-strictly to $y$.\smallskip
 
We turn now our attention to $T_1$.  Observe that $\displaystyle (iT_1)^{\#}(a)=(i\lim_{\lambda}x_{\lambda}\circ a^*)^*=-i\lim_{\lambda}x_{\lambda}\circ a=-iT_1(a)$. Moreover, since $\displaystyle iT_1(a)=i\lim_{\lambda}x_{\lambda}\circ a=\lim_{\lambda}ix_{\lambda}\circ a$ and $(ix_{\lambda})^*=-ix_{\lambda},$ it follows, as in the case of $T_2,$ that $iT_1$ is an anti-symmetric triple derivation on $\mathfrak{A},$ and hence there exists and anti-symmetric element $\widetilde{x}\in M(\mathfrak{A})$ satisfying $iT_1 (a)=\widetilde{x}\circ a.$  As a consequence, we have $T_1 (a)=(-i\widetilde{x})\circ a.$ Observe that the element $x:=-i\widetilde{x}$ is symmetric and $(x_{\lambda})\to x $ in the J-strict topology.\smallskip
  
Finally, %$$ T(a)=\lim_{\lambda} z_{\lambda}\circ a=\lim_{\lambda} x_{\lambda}\circ a + y_{\lambda}\circ a=x\circ a + y \circ a=(x+y)\circ a$$  
combining all the previous conclusions we get $(z_{\lambda}) = (x_{\lambda} +y_{\lambda})\to x+y $ J-strictly in $M(\mathfrak{A}).$ The rest is clear from Proposition~\ref{p strict density in the multipliers}.
\end{proof}

Let $X$ be a Banach space. It is known that every SOT-Cauchy net $(T_{\lambda})$ of bounded linear operators in $B(X)$ defines a linear mapping $T_0: X\to X$ such that $(T_{\lambda} (x))\to T_0(x)$ in norm for every $x\in  X$. However the linear mapping $T_0$ need not be continuous. If we additional assume that $X$ is a JB$^*$-triple or a JB$^*$-algebra, and each $T_{\lambda}$ is a triple derivation or a Jordan derivation, respectively, the mapping $T_0$ is a Jordan derivation or a triple derivation, and hence continuous \cite{BarFri90,HejNik96,PeRu2014}. However, it is not obvious whether, assuming that each $T_{\lambda}$ is a generalized Jordan derivation, the mapping $T_0$ is a generalized derivation too. In the proof of the previous theorem we relay on triple derivations (or on anti-symmetric Jordan derivations) with the aim of avoiding this obstacle. \smallskip

We establish now some direct consequences of our previous theorem. We recall first the notion of triple multiplier. Let $E$ be a JB$^*$-triple. The \emph{triple multipliers} of $E$ in $E^{**}$ is defined as the JB$^*$-triple $$M (E):=\{ x\in E^{**}: \{x, E, E\}\subseteq E\}\ \ \  \hbox{ (see \cite{BuChu92})},$$  which is the largest JB$^*$-subtriple of $E^{**}$ containing $E$ as a closed triple ideal (see \cite[Theorem 2.1]{BuChu92}). When a JB$^*$-algebra is regarded as a JB$^*$-triple, its triple multipliers and its multipliers as JB$^*$-algebra define the same objects (cf. \cite{BuChu92} or \cite[pages 42, 43]{GarPe2021JB}). For a C$^*$-algebra $A$ we have two ``strict topologies'' on $M(A)$, the next lemma shows that they coincide.

\begin{lemma}\label{l the Jordan and the strict topologies coincide on C*-algebras} Let $A$ be a C$^*$-algebra. Then the C$^*$-strict topology on $M(A)$ as C$^*$-algebra coincides with the J-strict topology when $A$ is regarded as a JB$^*$-algebra.   
\end{lemma}  

\begin{proof}  Since for each $x\in M(A)$ and each $a\in A$ we have $\| a \circ x \| \leq \frac12 (\|a x\| + \|x a\|)$, it is clear that the usual C$^*$-strict topology of $M(A)$ is stronger than the J-strict topology of $M(A)$ as JB$^*$-algebra. Reciprocally, we first observe that the C$^*$-strict topology of $M(A)$ is also given by the seminorms of the form $x\mapsto \| a x\|, \| x a\|$ ($x\in M(A)$) with $a$ running in $A_{sa},$ or simply among the positive elements in the closed unit ball of $A$. For $a\in A^{+},$ by recalling an identity from \cite[page 253]{BuChu92}, we have 
$$\begin{aligned} xa^2 &=\{x,a,a\}+\{x,a^{\frac{1}{2}},a^{\frac{1}{2}}\}a-a\{x,a^{\frac{1}{2}},a^{\frac{1}{2}}\} \\
		&= a^2 \circ x + (x\circ a ) a - a (x \circ a)
	\end{aligned}$$  (actually the second equality holds for every self-adjoint $a$),
 which implies that
	\begin{equation}\label{eq norm xa in terms of a circ x}
		\|x a^2\| \leq \|a^2 \circ x\| + 2 \|a\| \ \|a\circ x\|.
	\end{equation} Similarly, $$ a^2 x= a^2 \circ x - (x\circ a ) a + a (x \circ a),$$ and hence \begin{equation}\label{eq norm a x in terms of a circ x}
	\|a^2 x\| \leq \|a^2 \circ x\| + 2 \|a\| \ \|a\circ x\|.
\end{equation} Since every positive element $a\in A$ admits an square root, it follows from \eqref{eq norm xa in terms of a circ x} and \eqref{eq norm a x in terms of a circ x} that $$\lambda_a (x) = \|a x \|,  \rho_a (x) := \|x a\|\leq \|a \circ x\| + 2 \|a^{\frac12}\| \ \|a^{\frac12}\circ x\|,$$ for all $x\in M(A)$ and every positive $a\in A$. This proves that the J-strict topology of $M(A)$ as JB$^*$-algebra coincides with the C$^*$-strict topology. 
\end{proof}

Our next goal is to study the continuity with respect to the J-strict topology of the natural extension of an onto Jordan homomorphisms to the multipliers JB$^*$-algebra.

\begin{proposition}\label{p extensions of surjective Jordan homomorphisms} Let $\mathfrak{A}$ and $\mathfrak{B}$ be JB$^*$-algebras. Then 
every continuous surjective Jordan homomorphism $\Phi:\mathfrak{A} \to \mathfrak{B}$ extends to a J-strict continuous Jordan homomorphism $\tilde{\Phi}: M(\mathfrak{A})\to M(\mathfrak{B})$.  
\end{proposition}
 
\begin{proof} Let $\Phi^{**}:\mathfrak{A}^{**} \to \mathfrak{B}^{**}$ denote the bitransposed mapping of $\Phi$. Clearly $\Phi^{**}$ is a weak$^*$ continuous Jordan homomorphism (by the separate weak$^*$ continuity of the Jordan product of $\mathfrak{A}^{**}$), and, by the surjectivity of $\Phi$, $\Phi^{**}$ maps $M(\mathfrak{A})$ to $M(\mathfrak{B})$.  The desired extension is $\tilde{\Phi} = \Phi^{**}|_{M(\mathfrak{A})}$. Indeed, if $(x_{\lambda})\to x$ J-strictly in $M(\mathfrak{A})$, for each $b =\Phi (a)\in \mathfrak{B}$ (with $a\in \mathfrak{A}$), we have $$ (\Phi(x_{\lambda}) \circ b) = (\Phi(x_{\lambda}\circ a)) \to (\Phi(x \circ a)) = \Phi(x) \circ b\hbox{ in norm,}$$ which proves that $ (\Phi(x_{\lambda}))\to \Phi(x)$ J-strictly. 
\end{proof}

Let $\Psi: M(\mathfrak{A})\to M(\mathfrak{B})$ be a continuous Jordan homomorphism, where $\mathfrak{A}$ and $\mathfrak{B}$ are JB$^*$-algebras. Clearly, $\Psi$ is J-strict continuous when $\Psi (\mathfrak{A})\supseteq \mathfrak{B}$. However the latter condition is not necessary to get the J-strict continuity. 
For example, let $e$ be a partial isometry in $B(\ell_2) = M(K(\ell_2))$ satisfying $e^* e =\textbf{1}$ and $e e^*\neq \textbf{1}$. The mapping $\Psi : B(\ell_2)\to B(\ell_2)$, $\Psi (x) = e x e^*$ is a $^*$-homomorphism, $K(\ell_2)$ is not contained in $\Psi(K(\ell_2),)$ and it is easy to check that $\Psi$ is J-strict (strict) continuous (see \cite{McKenneyMathStackExch}). Furthermore, there exist $^*$-homomorphisms on $M(\mathfrak{A})$ vanishing on $\mathfrak{A}$ which are not J-strict continuous. 
Consider the following example from  \cite[Example 3.2.3]{FarahBook2000} or \cite{SleziakMathStackExch} $\mathfrak{A}= c_0$ with $M(\mathfrak{A}) = M(c_0) = \ell_{\infty}$. Let $(\mathcal{U}_n)$ be a sequence of nonprincipal ultrafilters. Consider the $^*$-homomorphism $\Phi :\ell_{\infty} \to \ell_{\infty}$ given by $\displaystyle \Phi(x)(n)=\lim_{k\to \mathcal{U}_n} x(k)\in \mathbb{C}$. Having in mind that each $\mathcal{U}_n$ is nonprincipal, it can be seen that $\Psi(c_0) =\{0\}$.  
In order to see that $\Phi$ is not strict continuous, let us consider the sequence $(e_m)_m$ with $e_m (n) = 1$ if $n\leq m$ and $e_m(n) =0$ otherwise. Clearly, $(e_m)\to \textbf{1}$ strictly in $\ell_{\infty}$. However, as we commented above, $\Phi(e_m) =0$ for all $m$ and $\Phi(\textbf{1})= \textbf{1}$, showing that $\Phi$ is not J-strict continuous. \smallskip

Let us assume that $u$ is a unitary in the multipliers algebra of a JB$^*$-algebra $\mathfrak{A}$. By noticing that $\mathfrak{A} \subseteq M(\mathfrak{A})$ and recalling the weak$^*$-density of $\mathfrak{A}$ in $\mathfrak{A}^{**}$ and the separate weak$^*$ continuity of the triple product of $\mathfrak{A}^{**},$ we conclude that $u$ is a unitary in the latter JBW$^*$-algebra. By applying that $u\in M(\mathfrak{A}),$ it follows that $\mathfrak{A}$ is closed for the Jordan product $\circ_u$ and the involution $*_{u}$, that is, $(\mathfrak{A}, \circ_{u}, *_{u})$ is a JB$^*$-subalgebra of the $u$-isotope $M(\mathfrak{A})(u)$.         

\begin{proposition}\label{p extensions of surjective triple homomorphisms} Let $\mathfrak{A}$ and $\mathfrak{B}$ be JB$^*$-algebras. Then every surjective triple homomorphism  $\Phi:\mathfrak{A} \to \mathfrak{B}$ between JB$^*$-algebras extends to a J-strict continuous triple homomorphism $\tilde{\Phi}: M(\mathfrak{A})\to M(\mathfrak{B})$.  
\end{proposition}

\begin{proof} As in the previous proposition, $\Phi^{**}:\mathfrak{A}^{**} \to \mathfrak{B}^{**}$ is a weak$^*$ continuous triple homomorphism thanks to the separate weak$^*$ continuity of the triple product of the JBW$^*$-algebras $\mathfrak{A}^{**}$ and $\mathfrak{B}^{**}$. We claim that $\Phi^{**} (M(\mathfrak{A})) \subseteq M(\mathfrak{B})$. Indeed, since $M(\mathfrak{A})$ and $M(\mathfrak{B})$ coincide with the triple multipliers of $\mathfrak{A}$ and $\mathfrak{B},$ respectively, $\Phi$ is onto and $\Phi^{**}$ is a triple homomorphism, we have $$\{\mathfrak{B}, \mathfrak{B}, \Phi^{**} (M(\mathfrak{B}))\} = \{\Phi(\mathfrak{A}), \Phi(\mathfrak{A}), \Phi^{**} (M(\mathfrak{B}))\} =\Phi \{\mathfrak{A}, \mathfrak{A}, M(\mathfrak{A})\} \subseteq \Phi (\mathfrak{A}) =\mathfrak{B} $$ which proves the claim.  \smallskip
	
The element $u=\Phi^{**} (\textbf{1})$ must be a tripotent in $M(\mathfrak{B})$. The surjectivity of $\Phi$ also implies that $u$ is a unitary in $\mathfrak{B}^{**}.$ Let $\mathfrak{B} (u)$ denote the JB$^*$-algebra $\mathfrak{B}$ with Jordan product $\circ_u$ and involution $*_{u}$.  Since $\Phi^{**}: \mathfrak{A}^{**}\to \mathfrak{B}^{**}(u)$ is a unital triple homomorphism, the mapping $\Phi : \mathfrak{A}\to \mathfrak{B} (u)$ is a surjective Jordan $^*$-homomorphism. Proposition~\ref{p extensions of surjective Jordan homomorphisms} assures the existence of an extension of $\Phi$ to a J-strict continuous Jordan homomorphism $\tilde{\Phi}: M(\mathfrak{A})\to M(\mathfrak{B}(u))$. Clearly, $\tilde{\Phi}$ is a triple homomorphism by the uniqueness of the triple product.\smallskip

To finish our argument, let $(x_j)$ be a net in $M(\mathfrak{A})$ converging to some $x$ of $M(\mathfrak{A})$ in the J-strict topology.  It follows from the previous conclusion that $b\circ_u \tilde{\Phi}(x_j) \to b\circ_u \tilde{\Phi}(x)$ in norm, for each $b\in \mathfrak{B}$. By \eqref{eq uniqueness of the triple product} the net 
$$\{b,c,\tilde{\Phi}(x_j)\} = (b\circ_u c^*)\circ_u \tilde{\Phi}(x_j) +  (c^*\circ_u \tilde{\Phi}(x_j)) \circ_u b - (b\circ_u \tilde{\Phi}(x_j)) \circ_u c^*$$ tends to $(b\circ_u c^*)\circ_u \tilde{\Phi}(x) +  (c^*\circ_u \tilde{\Phi}(x)) \circ_u b - (b\circ_u \tilde{\Phi}(x)) \circ_u c^* =\{b,c,\tilde{\Phi}(x)\}$ in norm. %Similarly, $\{b,\Phi(x_j),c\}\to \{b,\Phi(x),c\}$ in norm.  
A new application of \eqref{eq uniqueness of the triple product} for the original Jordan product of $\mathfrak{B}$ proves that $$b^2 \circ \tilde{\Phi}(x_j) = \{b,b,\tilde{\Phi}(x_j)\} \to  \{b,b,\tilde{\Phi}(x)\} = b^2 \circ \tilde{\Phi} (x)$$ in norm for every $b\in \mathfrak{B}_{sa}$. Therefore, $\tilde{\Phi}(x_j)\to \tilde{\Phi}(x)$ in the strict topology of $M(\mathfrak{B})$.
\end{proof}

\begin{remark}\label{r u-isotopes}{\rm Let $\mathfrak{A}$ be a JB$^*$-algebra. Let $u$ be any unitary in the multipliers algebra $M(\mathfrak{A})$. We have proved in the arguments leading to the previous proposition that $\mathfrak{A}$  is a JB$^*$-subalgebra of the $u$-isotope $M(\mathfrak{A})(u)$. We denote the corresponding JB$^*$-algebra by $\mathfrak{A}(u)$. We have also justified that $M(\mathfrak{A}(u)) = M(\mathfrak{A})$ with $S(M(\mathfrak{A}), \mathfrak{A}(u)) = S(M(\mathfrak{A}), \mathfrak{A})$.}
\end{remark}

Elements $a,b$ in a Jordan algebra $\mathfrak{A}$ are said to \emph{operator commute} if $$ a\circ (b\circ x)=(a\circ x)\circ b$$ for every $x\in \mathfrak{A}.$ By the mentioned \emph{Macdonald's theorem} or by the \emph{Shirshov-Cohn theorem} \cite[Theorem 2.4.14]{HOS}, it can be easily checked that a couple of hermitian elements $h,k$ in a JB$^*$-algebra $\mathfrak{A}$ operator commutate if and only if they operator commute relative to any Jordan algebra of self-adjoint operators containing them, if and only if $h^2 \circ k = U_h (k)$ (cf. \cite[Proposition 1]{Topping}). The center of $\mathfrak{A}$ %, $Z(\mathfrak{A})$, 
is the set of all elements $z$ in $\mathfrak{A}$ such that $z$ and $b$ operator commute for every $b$ in $\mathfrak{A}$. Elements in the center are called central.\smallskip

Let $T : \mathfrak{A}\to E$ be a surjective bounded linear operator preserving orthogonality from a JB$^*$-algebra to a JB$^*$-triple. Let $h=T^{**} (1)\in E^{**}$ and let $r$ denote the range tripotent of $h$ in $E^{**}$. By \cite[Proposition 2.7]{GarPe2021JB} there exists a Jordan $^*$-homomorphism  $S: \mathfrak{A}\to E_2^{**}(r)$ such that $S(x)$ and $h$ operator commute in the JB$^*$-algebra $E_2^{**}(r)$ and \begin{equation}\label{eq fund equation conts OP JBSTAR}
	T(x)=h\circ_{r} S(x)= \{h,r,S(x)\}= U_{h^{\frac12}} (S(x)),
\end{equation} for every $x\in M(\mathfrak{A})$ where $h^{\frac12}$ is the square root of the positive element $h$ in the JB$^*$-algebra $E_2^{**} (r)$ and the $U$ operator is the one given by this JB$^*$-algebra. Furthermore,  the following statements hold:
\begin{enumerate}[$(a)$]
	\item $r$ is a unitary in $E^{**};$
	\item $h$ belongs to $M(E)$;
	\item $h$ is invertible (and positive) in the JBW$^*$-algebra ${E}^{**}= E^{**}_2 (r)$;
	\item $r$ belongs to $M(E)$, and consequently $E$ is a JB$^*$-algebra;
	\item The triple homomorphism $S$ is $E$-valued and surjective;
	\item If $x\in M(\mathfrak{A})$ then $T^{**}(x)\in M(E) $.
	%\item The quotient mapping $\widehat{T}: \mathcal{A}/\ker(T)\to E,$ $\widehat{T}(x+\ker(T)) = T(x)$ is an orthogonality preserving bounded linear bijection;
	%\item There exist a triple homomorphism $S: \mathcal{A}\to E$ and a triple isomorphism $\widehat{S}: \mathcal{A}/\ker(S)\to E$ satisfying: \begin{enumerate}[$(1)$] \item $\ker(T)= \ker(S)$;
		%\item $\widehat{S}^{**} (1+\overline{\ker(S)}^{w^*}) =S^{**} (1) = r;$
		%\item $\widehat{S} (x+{\ker(S)}) =S(x)$;
		%\item $\widehat{S}(x)$ and $h$ operator commute in $E^{**}_2(r)$, for all $x\in \mathcal{A};$
		%\item $\widehat{T}(x+\ker(T)) =T(x) = h \circ_r \widehat{S}(x),$ %$=U_{h^{\frac12}} S(x),$
		%for all $x\in \mathcal{A}$.%, where $h^{\frac12}$ denotes the square root of the positive element $h$ in the JB$^*$-algebra $E_2^{**} (r)$.
	%\end{enumerate}
\end{enumerate}

The next corollary is a consequence of  Proposition~\ref{p extensions of surjective triple homomorphisms} and the just commented results.

\begin{corollary}\label{c extension of OP }  Let $\mathfrak{A}$ and $\mathfrak{B}$ be JB$^*$-algebras.
Then every surjective orthogonality preserving operator $\Phi:\mathfrak{A} \to \mathfrak{B}$ extends to a J-strict continuous orthogonality preserving operator $\tilde{\Phi}: M(\mathfrak{A)}\to M(\mathfrak{B})$.
\end{corollary}

\begin{proof} By the conclusions in \cite[Proposition 2.7]{GarPe2021JB} commented before the corollary, there exist a surjective triple homomorphism $S: \mathfrak{A}\to \mathfrak{B}$, an element $h\in M(\mathfrak{B})$ whose range tripotent $r=r(h)$ is a unitary in $M(\mathfrak{B})$, $h$ operator commute with each element in $\mathfrak{B}$ in the $u$-isotope $M(\mathfrak{B})_2(r)$ and $\Phi (x) = h\circ_r S(x)$ for all $x\in \mathfrak{A}$. By Proposition~\ref{p extensions of surjective triple homomorphisms} the triple homomorphism $S$ admits a J-strict continuous extension to a triple homomorphism $\tilde{S} : M(\mathfrak{A})\to M(\mathfrak{B})$. We only have to show that the mapping $z\mapsto h\circ_r z$ is J-strict continuous on $M(\mathfrak{B})$, but this follows easily from the fact that $h$ is a central element in $M(\mathfrak{B})_2(r)$ and Remark~\ref{r u-isotopes}.
\end{proof}

Let us finish this section with a geometric observation and an open question.

\begin{remark}\label{r final comments with seminorms related to Hilbertian seminorms}{\rm Let $a,x$ be elements in a C$^*$-algebra $A$ with $a$ self-adjoint. By the Gelfand-Naimark axiom we have 
\begin{equation}\label{eq inequality seminorm left with Groth seminorm} 
\| a x\|^2 = \| a x x^* a\| \leq  2 \| a (x \circ x^*) a\|\leq \|a x x^*a\| + \|a x^* x a\| = \| a x\|^2 + \| x a \|^2.
\end{equation} 
Similarly, 
\begin{equation}\label{eq inequality seminorm right with Groth seminorm} 
\| x a \|\leq  \sqrt{2} \ \| a (x \circ x^*) a\|^{\frac12} =\sqrt{2} \ \| U_{a} (x \circ x^*) \|^{\frac12}.
\end{equation} 
We therefore deduce that the C$^*$-strict topology of $M(A)$ coincides with the topology generated by the seminorms of the form $x\mapsto   \| a (x \circ x^*) a\|^{\frac12}$ ($x\in M(A)$) with $a$ in $A_{sa}.$ We observe that the latter seminorms are very closed to the preHilbertian seminorms appearing in Grothendieck's inequalities for C$^*$-algebras (cf. \cite{HKPP-BF,KaPePfi22, pisier2012grothendieck}).\smallskip 
	
We have already employed one of the fundamental inequalities in JB$^*$-algebra theory saying that \begin{equation}\label{inequelity in HOS} 
\| a\circ b\|^2 \leq \| a\| \|U_b (a)\|,
\end{equation} 
for every couple of self-adjoint elements $a,b$ in a JB$^*$-algebra with $a$ positive (see \cite[Lemma 3.5.2~$(ii)$]{HOS}).\smallskip

Let $\mathfrak{A}$ be a JB$^*$-algebra.  It is natural to consider the topology on $M(\mathfrak{A})$ generated by all seminorms of the form $x\mapsto   \| U_{a} (x \circ x^*) \|^{\frac12}$ ($x\in M(\mathfrak{A})$) with $a$ in $\mathfrak{A}_{sa}.$ What we can prove is that this topology is stronger than the J-strict topology of $M(\mathfrak{A})$. Indeed, given given any two self-adjoint elements $a,x$ in a JB$^*$-algebra $\mathfrak{B}$, let $\mathfrak{C}$ denote the JB$^*$-subalgebra of $\mathfrak{B}$ generated by $a$ and $x$.  By the Shirshov-Cohn or Macdonald theorems (see \cite[2.4.14 and 2.4.15]{HOS} or \cite[Corollary 2.2]{Wright77}), there exists a C$^*$-algebra $A$ containing $\mathfrak{C}$ as a JB$^*$-subalgebra. Working in $\mathfrak{C}$ and $\mathfrak{A}$, we deduce from \eqref{eq inequality seminorm left with Groth seminorm} and \eqref{eq inequality seminorm right with Groth seminorm} that \begin{equation}\label{eq inequ bounde for the norm of a circ x with Grothendieck seminorm} \| a\circ x\| \leq\frac12 \left(\| a x \| +\| xa \|\right) \leq \sqrt{2} \ \| U_{a} (x \circ x^*) \|^{\frac12},
\end{equation} which seems to be an alternative inequality to the useful one in \cite[Lemma 3.5.2~$(ii)$]{HOS}.
Consequently, for $a\in \mathfrak{A}_{sa}$ and $z\in M(\mathfrak{A})$ we write $z = x+i y$ with $x,y\in \mathfrak{A}_{sa}$ and thus $$\| a\circ z \| \leq \| a\circ x \| + \|a\circ y \| \leq \sqrt{2} \ \| U_{a} (x \circ x^*) \|^{\frac12} + \sqrt{2} \ \| U_{a} (y \circ y^*) \|^{\frac12}\leq 2 \sqrt{2} \ \| U_{a} (z \circ z^*) \|^{\frac12},$$ which proves the desired statement. 
The question whether the topology on $M(\mathfrak{A})$ generated by all seminorms of the form $x\mapsto   \| U_{a} (x \circ x^*) \|^{\frac12}$ ($x\in M(\mathfrak{A})$) with $a$ in $\mathfrak{A}_{sa}$ coincides with the J-strict topology remains as an open challenge. }
\end{remark}

Let us recall the definition of the preHilbertian seminorm appearing in Grothendieck's\hyphenation{Grothen-dieck} inequalities for JB$^*$-algebras. For each positive normal functional $\phi$ in the predual of a JBW$^*$-algebra $\mathfrak{A}$, the sesquilinear form $(x,y)\mapsto \phi (x\circ y^*)$ is positive, and gives rise to a preHilbertian seminorm $\|a\|_{\phi}^2 := \phi (x\circ x^*)$ on $\mathfrak{A}$.  The \emph{strong$^*$ topology} of $\mathfrak{A}$ is the topology generated by all the seminorms $\|\cdot \|_{\phi}$ with $\phi$ running among all normal states of $\mathfrak{A}$ (cf. \cite{BarFri90}). In case that $\mathfrak{A}$ is a von Neumann algebra, this strong$^*$ topology is precisely the usual C$^*$-algebra strong$^*$ topology (see \cite[\S 3]{BarFri90}). The strong$^*$ topology is stronger than the weak$^*$ topology of $\mathfrak{A}$ (see \cite[Theorem 3.2]{BarFri90}) and the product of $\mathfrak{A}$ is jointly strong$^*$ continuous on bounded sets (see \cite[THEOREM in page 103]{RodPa91} or \cite[Proposition 2.4]{AlfsenShultz2003}).

\section{The topological strict dual of the multipliers algebra}\label{sec:3 Taylor topological dual of Multipliers with the strict topology}

Let $\mathfrak{A}$ be a JB$^*$-algebra. In this section we study the space of all J-strict continuous functionals on $M(\mathfrak{A})$ with the topology of uniform convergence on J-strict bounded subsets of $M(\mathfrak{A})$.\smallskip

In \cite{AkkLaa95,AkkLaa96} M. Akkar and M. Laayouni established a version of Cohen's factorization theorem for Jordan-Banach algebras admitting a bounded approximate identity. The arguments can be adapted to get a result on factorizations in Jordan modules. It seems worth to devote a few lines to recall the definition and the basic properties of a bounded approximate identity in a JB$^*$-algebra. An (increasing) \emph{approximate identity} or an \emph{approximate unit} in a JB$^*$-algebra $\mathfrak{A}$ is an increasing net $(e_{\lambda})$ in $\mathfrak{A}$ satisfying $0 \leq  e_{\lambda}$, $\|e_{\lambda}\| \leq 1$ for all $\lambda$ and $\displaystyle \lim_{\lambda} \|a\circ e_{\lambda} -a\| =0$ for all $a$ in $\mathfrak{A}$ (cf. \cite[Definition 1.29]{AlfsenShultz2003} or \cite[3.5.1]{HOS}). Every JB$^*$-algebra $\mathfrak{A}$ admits an (increasing) approximate identity $(e_{\lambda})$ (see \cite[Lemma 4]{Ed80}, \cite[Lemma 1.32]{AlfsenShultz2003} or \cite[Proposition 3.5.4]{HOS}). It is further known that $\displaystyle \lim_{\lambda} \| U_{\textbf{1}-e_{\lambda}} (a)\| =0$ for all $a\in \mathfrak{A}$ (cf. \cite[comments before Lemma 4]{Ed80}), which implies that $\|e_{\lambda}^2 \circ a - a\|\to 0$. For each positive functional $\varphi$ in the dual space, $\mathfrak{A}^*$, of $\mathfrak{A}$ we have $\displaystyle \|\varphi \| =\lim_{\lambda} \varphi (e_{\lambda})$ (see \cite[Proof of Lemma 3.6.5]{HOS}).\smallskip

The dual of $\mathfrak{A}$ is a Jordan-Banach module with its natural norm and the Jordan module product defined by 
\begin{equation}\label{eq Jordan product on the dual}
(\phi\circ a) (b) = \phi (a\circ b)\ \ \ \ (\phi\in \mathfrak{A}^*, a,b\in \mathfrak{A}).	
\end{equation}	

Let $\varphi$ be a non-zero positive functional in $\mathfrak{A}^*$, let $(e_{\lambda})$ be an approximate identity in $\mathfrak{A}$ and let $\textbf{1}$ stand for the unit in $\mathfrak{A}^{**}$.  By the Cauchy-Schwarz inequality we have $$\varphi (e_{\lambda}) =\varphi (e_{\lambda}\circ \textbf{1})  \leq \varphi (e_{\lambda}^2)^{\frac12} \varphi (\textbf{1})^{\frac12} = \varphi (e_{\lambda}^2)^{\frac12} \|\varphi\|^{\frac12} \leq \|\varphi\|,$$ witnessing that $\displaystyle \lim_{\lambda} \varphi (e_{\lambda}^2) = \|\varphi\|.$ A new application of the Cauchy-Schwarz inequality gives $$|\varphi (a\circ e_{\lambda} -a)|^2 =|\varphi (a\circ (\textbf{1}- e_{\lambda}))|^2 \leq \varphi (a\circ a^*) \varphi ((\textbf{1}- e_{\lambda})^2)\leq \|\varphi\| \varphi ((\textbf{1}- e_{\lambda})^2), $$ uniformly in $\|a\|\leq 1$. Therefore \begin{equation}\label{eq elambda is an approximate unit in the dual}
\lim_{\lambda} \|\varphi \circ e_{\lambda} -\varphi\| = \lim_{\lambda} \sup_{\|a\|\leq 1} |\varphi (a\circ e_{\lambda} -a)|\leq \|\varphi\|^{\frac12}  \lim_{\lambda} \varphi ((\textbf{1}- e_{\lambda})^2)^{\frac12} =0. 
\end{equation} Similarly,  \begin{equation}\label{eq elambda2 is an approximate unit in the dual}
\lim_{\lambda} \|\varphi \circ e_{\lambda}^2 -\varphi\|  =0. 
\end{equation} Clearly every functional $\phi\in \mathfrak{A}^*$ writes in the form $\phi = \phi_1 + i \phi_2$ with $\phi_1 = \frac{\phi+\phi^{\sharp}}{2}$ and $\phi_2 = \frac{\phi-\phi^{\sharp}}{2 i }$ symmetric, and hence $\phi_j$ can be regarded as a functional in the dual space of the JB-algebra $\mathfrak{A}_{sa}$. 
A Hahn-type decomposition of functionals in $\mathfrak{A}_{sa}^*$ asserts that each $\phi_j$ writes in the form $\phi_j = \phi_j^+- \phi_j^{-}$ with $\phi_{j}^{\pm}$ positive (see \cite[$(A26)$]{AlfsenShultz2003} or \cite[Lemma 1.2.6]{HOS}). 
Therefore, every functional in $\mathfrak{A}^*$ writes as a linear combination of four positive functionals, which together with \eqref{eq elambda is an approximate unit in the dual} and \eqref{eq elambda2 is an approximate unit in the dual} proves  \begin{equation}\label{eq elambda genral is an approximate unit in the dual}
\lim_{\lambda} \|\phi \circ e_{\lambda}^2 -\phi\| =  \lim_{\lambda} \|\phi \circ e_{\lambda} -\phi\| =0 
\end{equation}  for all $\phi\in \mathfrak{A}^{*}$ and every approximate unit $(e_{\lambda})$ in $\mathfrak{A}$.\smallskip

Let $M$ be a Jordan-Banach $\mathfrak{A}$ module. The space $\mathfrak{A} \oplus M$ can be
equipped with a structure of Jordan-Banach algebra with respect
to the product $$(a+x)\circ (b+y)=a\circ b+b\circ x
+a\circ y,$$ and norm given by $\mathfrak{A}\oplus M$ given by $\|a+x\|=\|a\|+\|x\|.$ This structure is known as the \emph{Jordan split-null extension} of $M$ and $\mathfrak{A}$ (cf. \cite{HejNik96}).

\begin{proposition}\label{p Cohen factorization for modules} Let $M$ be a Jordan-Banach module over a JB$^*$-algebra $\mathfrak{A}$. Let $(e_{\lambda})$ be an approximate identity for $\mathfrak{A}$. Suppose that $\displaystyle\lim_{{\lambda}}e_{\lambda}\circ x=x$ holds for every $x\in M.$
	Then for each $z\in M$ there exists $a\in \mathfrak{A}$ and $y\in M$ satisfying $z = U_{a} (y)$.  
\end{proposition}  

\begin{proof}
	$\mathfrak{A}\oplus M$ is a Jordan-Banach algebra with an approximate identity. Indeed, fix $a+x \in \mathfrak{A}\oplus M.$ We have 
	$$\lim_{{\lambda}} \| (a+x)\circ e_{\lambda}-(a+x)\|_{_{\mathfrak{A}\oplus M}}=\lim_{{\lambda}} \|a\circ e_{\lambda}- a\|_{\mathfrak{A}}+\| e_{\lambda}\circ x-x\|_{M}   =0$$
	Showing that $(e_{\lambda})$ is a bounded approximate identity for $\mathfrak{A} \oplus M$.
	Consequently, $\mathfrak{A} \oplus M$ satisfies Cohen's factorisation property (see \cite{AkkLaa95,AkkLaa96}).  Let us fix $z\in M.$ By \cite[Theorem II.2]{AkkLaa96} there exist $a+x,b+y\in \mathfrak{A}\oplus M$ such that $z=U_{a+x}(b+y).$ Moreover, the element $b+y$ lies in the norm closure of  ${U_{(\mathfrak{A}\oplus M)'}(z)},$ where $(\mathfrak{A}\oplus M)'$ stands for the unitization of $\mathfrak{A}\oplus M.$ Straightforward computations shows that $U_{(\mathfrak{A}\oplus M)'}(z)\subseteq M$. Thus $b+y\in \overline{U_{(\mathfrak{A}\oplus M)' }(z)}\subseteq M,$ and hence $b=0$. Finally, we have 
	$$z=\{a+x,y,a+x\}=\{a,y,a\} .$$\end{proof}

Now, by combining Proposition~\ref{p Cohen factorization for modules}  with \eqref{eq elambda genral is an approximate unit in the dual} we get:

\begin{corollary}\label{c dual factorizes} Let $\mathfrak{A}$ be a JB$^*$-algebra. Then for each $\phi \in\mathfrak{A}^*$ there exit $\varphi \in \mathfrak{A}^{*}$ and $a\in \mathfrak{A}$ such that $\phi = \varphi U_a = 2 (\varphi \circ a)\circ a - \varphi \circ a^2$.  
\end{corollary}

Let us focus next on the following technical result. We recall first that, by the Jordan identity, for every couple of elements $a,b$ in a JB$^*$-triple $E$ the mapping $\delta (a,b) = L(a,b) -L(b,a)$ is a triple derivation on $E$. We recall that if $k$ is a skew symmetric element in a unital JB$^*$-algebra $\mathfrak{B}$, the mapping $\delta (\frac12 k, \textbf{1}) = L(\frac12 k,\textbf{1}) - L(\textbf{1}, \frac12 k) = M_{k}$ is a triple derivation on $\mathfrak{B}$ (cf. \cite[Lemma 2]{HoMarPeRu}). Therefore, for each skew symmetric element $k$ in an arbitrary JB$^*$-algebra $\mathfrak{A}$, the mapping $M_k : a\mapsto k\circ a$ ($a\in \mathfrak{A}$) is a triple derivation on $\mathfrak{A}$.    

\begin{lemma}\label{l strict convergence of Jordan times an approximate unit} Let $(e_{\lambda})$ be an approximate identity in a JB$^*$-algebra $\mathfrak{A}$. Then for each $x\in M(\mathfrak{A})$ the net $(e_{\lambda}\circ x)_{\lambda}$ converges to $x$ in the J-strict topology of $M(\mathfrak{A})$.
\end{lemma}	 

\begin{proof} As we have seen in the comments preceding this lemma, the mapping $M_{i e_{\lambda}}$ is a triple derivation on $M(\mathfrak{A})$. Therefore, for each $a\in \mathfrak{A}_{sa}$ we have 
$$i e_{\lambda}\circ \{a,a,x\} =  \{i e_{\lambda}\circ a,a,x\} +  \{a, i e_{\lambda}\circ a,x\} +  \{a,a, i e_{\lambda}\circ x\}. $$ Let us pay attention to each summand in the previous identity. Since $\{a,a,x\}, a\in \mathfrak{A}$ and  $(e_{\lambda})$ is an approximate unit, it follows that the nets $(e_{\lambda}\circ \{a,a,x\}),$ $( \{e_{\lambda}\circ a,a,x\})$, and  $(\{a,e_{\lambda}\circ a,x\})$ converge to $\{a,a,x\}$ in norm. Therefore,  $a^2 \circ (e_{\lambda}\circ x) =  \{a,a, e_{\lambda}\circ x\}\to \{a,a, x\} = a^2 \circ x$ in norm. A standard polarization argument, combined with the existence of square roots of positive elements, shows that $(e_{\lambda}\circ x)_{\lambda}$ converges to $x$ in the J-strict topology of $M(\mathfrak{A})$.
\end{proof}

\begin{remark}\label{r Lemma 3.3 is valid for striclty convergent nets of sa elements with limit 1}{\rm The proof given above is actually valid to obtain the following conclusion:}  Let $(a_{\lambda})$ be a net of self-adjoint elements in a JB$^*$-algebra $\mathfrak{A}$ converging to $y\in M(\mathfrak{A})_{sa}$ in the J-strict topology of $M(\mathfrak{A})$. Then for each $x\in M(\mathfrak{A})$ the net $(a_{\lambda}\circ x)_{\lambda}$ converges to $y \circ x$ in the J-strict topology of $M(\mathfrak{A})$.
\end{remark}

Our next result, which is a Jordan version of \cite[Corollary 2.2]{Tay70}, offers a good understanding of J-strict continuous functionals on the multipliers algebra.  Let us make an observation first, by the Hahn-Banach theorem, every functional $\phi$ in the dual of a JB$^*$-algebra $\mathfrak{A}$ admits a norm-preserving linear extension to a functional $\widehat{\phi}$ in $M(\mathfrak{A})^*$. However, nothing guarantees that $\widehat{\phi}$ is  J-strict continuous.

\begin{proposition}\label{p dual of M(A) with strict topology} Let $\mathfrak{A}$ be a JB$^*$-algebra. Then a functional $\phi : M(\mathfrak{A})\to \mathbb{C}$ is J-strict continuous if and only if there exist  $a\in \mathfrak{A}$ and $\varphi\in \mathfrak{A}^*$ such that $\phi (x) = \varphi \left( U_a (x) \right)$ for all $x\in M(\mathfrak{A})$. Consequently, every functional $\phi\in  \mathfrak{A}^*$ admits a linear extension to a J-strict continuous functional $\widehat{\phi}: M(\mathfrak{A})\to \mathbb{C}$.
\end{proposition}

\begin{proof} Clearly, given $a\in \mathfrak{A}$ and $\varphi\in \mathfrak{A}^*$ the functional $\phi : M(\mathfrak{A})\to \mathbb{C}$, $\phi (x) = \varphi \left( U_a (x) \right)$ is J-strict continuous. Reciprocally, suppose that $\phi : M(\mathfrak{A})\to \mathbb{C}$ is J-strict continuous. Since the norm topology is stronger than the J-strict topology the functional $\phi$ lies in $M(\mathfrak{A})^*$, and hence its restriction $\phi|_{\mathfrak{A}}$ belongs to $\mathfrak{A}^*$.  Find, via Corollary~\ref{c dual factorizes}, $a\in \mathfrak{A}$ and $\varphi \in \mathfrak{A}^*$ such that $\phi|_{\mathfrak{A}}= \varphi U_{a}$. Let $(e_{\lambda})$ be an approximate identity in $\mathfrak{A}$. The identity $$\phi (e_{\lambda}\circ x)  =   \phi|_{\mathfrak{A}} (e_{\lambda}\circ x) = \varphi U_{a} (e_{\lambda}\circ x),$$ holds for all $x\in M(\mathfrak{A})$ and every $\lambda$. Lemma~\ref{l strict convergence of Jordan times an approximate unit} implies that $(e_{\lambda}\circ x)_{\lambda}$ converges to $x$ in the J-strict topology, and since $\phi$ is J-strict continuous the left hand side in the previous identity tends to $\phi (x)$. On the right hand side we apply that $$ U_{a} (e_{\lambda}\circ x)  = 2 a\circ (a\circ (e_{\lambda}\circ x)) - a^2 \circ (e_{\lambda}\circ x)$$ is a net in $\mathfrak{A}$ converging in norm to  $U_{a} (x)$ because  $(e_{\lambda}\circ x)\to x$ in the J-strict topology. Therefore $\phi (x) =   \varphi U_{a} (x)$, for all $x\in M(\mathfrak{A})$, which concludes the proof.	The last statement is clear from the factorization in Corollary~\ref{c dual factorizes}.
\end{proof}

Let us take a J-strict bounded set $F$ in the multipliers algebra of a JB$^*$-algebra $\mathfrak{A}$, that is, for each $a\in \mathfrak{A}$ the set $\{\|a \circ x\| : x\in F\}$ is bounded, equivalently, the set $\{M_x : x\in F\}$ is point-norm bounded when each $M_x$ is regarded as a bounded linear operator on $\mathfrak{A}$. It follows from the uniform boundedness principle that $\{M_x : x\in F\}$ is norm bounded when regarded in $B(\mathfrak{A})$. In particular,  J-strict bounded sets and norm bounded sets in $M(\mathfrak{A})$ coincide, and the topology on $(M (\mathfrak{A}), S(M(\mathfrak{A}),\mathfrak{A}) )^*$ of uniform convergence on J-strict bounded subsets of $M(\mathfrak{A})$ is precisely the norm topology. We can finally describe the dual of $(M(\mathfrak{A}), S(M(\mathfrak{A}),\mathfrak{A}))$, a result which generalizes Taylor's theorem in \cite{Tay70}.  

\begin{theorem}\label{t dual of multipliers isomorphic to the dual of A} Let $\mathfrak{A}$ be a JB$^*$-algebra. Then the mapping $$R: (M(\mathfrak{A}), S(M(\mathfrak{A}),\mathfrak{A}))^*\to \mathfrak{A}^*, \ \  \phi \mapsto \phi|_{\mathfrak{A}}$$ is an isomorphism and a surjective isometry.  
\end{theorem} 

\begin{proof} We have already commented that the topology on $(M (\mathfrak{A}), S(M(\mathfrak{A}),\mathfrak{A}) )^*$ of uniform convergence on J-strict bounded subsets of $M(\mathfrak{A})$ agrees with the norm topology. The mapping $R$ is clearly well defined and linear. The injectivity of $R$ follows from Proposition~\ref{p strict density in the multipliers}, while the surjectivity is a consequence of Proposition~\ref{p dual of M(A) with strict topology} (see also Corollary~\ref{c dual factorizes}). Finally, given $\phi\in  (M(\mathfrak{A}), S(M(\mathfrak{A}),\mathfrak{A}))^*$, we have seen in the proof of Proposition~\ref{p dual of M(A) with strict topology} that for each approximate identity $(e_{\lambda})$ in $\mathfrak{A}$ and $x\in M(\mathfrak{A})$, the net $(\phi (x\circ e_{\lambda}))$ converges to $\phi(x)$. Clearly, $x\circ e_{\lambda}\in \mathfrak{A}$ with $\|x\circ e_{\lambda}\|\leq \|x\|$, and thus $| \phi (x\circ e_{\lambda}) |\leq \| \phi|_{\mathfrak{A}} \|  \|x\| = \|R(\phi)\|  \|x\|$, witnessing that $\|\phi\|\leq \|R(\phi)\|\leq \|\phi\|$. 
\end{proof}

We finish this section with a discussion on J-strict equicontinuous families of functionals on $M(\mathfrak{A})$. If $A$ is a C$^*$-algebra, a result due to Taylor proves that a family $\mathcal{F}$ of strict continuous functionals in $M(A)^*$ is strict equicontinuous if and only if it is norm-bounded and for any approximate identity $(e_{\lambda})$ in $A$ we have $$\|\varphi- \varphi e_{\lambda} - e_{\lambda}\varphi + e_{\lambda} \varphi e_{\lambda}\| = \|\varphi - \varphi U_{e_{\lambda}}  \|\to 0,$$ uniformly on $\varphi \in \mathcal{F}$ \cite[Theorem 2.6]{Tay70}. In the setting of JB$^*$-algebras we can prove the necessary condition.

\begin{proposition}\label{p necesary condition for strict equicontinuity} Let $\mathcal{F}$ be a family of functionals in $(M(\mathfrak{A}), S(M(\mathfrak{A}), \mathfrak{A}))^*$, where $\mathfrak{A}$ is a JB$^*$-algebra. Suppose $\mathcal{F}$ is J-strict equicontinuous and $(e_{\lambda})$ is an approximate identity in $\mathfrak{A}$. Then $\mathcal{F}$ is norm bounded and $ \|\varphi - \varphi U_{e_{\lambda}}  \|\to 0,$  uniformly on $\varphi \in \mathcal{F}.$
\end{proposition} 

\begin{proof} Since every J-strict open neighbourhood of $0$ contains a norm open neighbourhood, the equicontinuity of the family $\mathcal{F}$ implies that it is norm bounded.  Let $(e_{\lambda})$ be an approximate identity in $\mathfrak{A}$. Since $\mathcal{F}$ is J-strict equicontinuous there exist positive elements $a_1,\ldots, a_m\in \mathfrak{A}$ and $\delta>0$ such that $|\varphi (x)|< 1$ for all $\varphi \in \mathcal{F}$ when $x\in\{ y \in M(\mathfrak{A}): \| a_i \circ y\| < \delta, \ \forall i =1,\ldots, m \}$ --let us recall that every element in $\mathfrak{A}$ is a linear combination of four positive elements. It is standard to check that, under these conditions, we have 
$\displaystyle  |\varphi ( x ) |\leq \delta^{-1} \sum_{i=1}^m \| x\circ a_i\|,$
 for all $x\in M(\mathfrak{A})$, $\varphi\in \mathcal{F}$. 	Therefore, given $x\in M(\mathfrak{A})$ with $\|x\|\leq 1$ and $\varphi\in \mathcal{F}$ we have \begin{equation}\label{eq first estimation of the norm of varphi minus varphi Ulambda} \left|\left(\varphi -  \varphi U_{e_{\lambda}} \right) (x) \right| = \left|\varphi U_{\textbf{1}-e_{\lambda}} (x) \right| \leq \delta^{-1} \sum_{i=1}^m \| U_{\textbf{1}-e_{\lambda}} (x) \circ a_i\|.
 \end{equation} In order to bound each summand on the right hand side of the previous inequality we first observe that we can assume without loss of generality that $x\in M(\mathfrak{A})_{sa}$ with $\|x\|\leq 1$. We also recall a celebrated inequality affirming that $$\| c\circ d \|^2 \leq \|c\| \|U_{d} (c)\|,$$  for every couple of self-adjoint elements $c,d$ in a JB$^*$-algebra with $c$ positive. By applying this inequality with $c = a_i$ and $d =  U_{\textbf{1}-e_{\lambda}} (x)$ we get $$\begin{aligned} \| U_{\textbf{1}-e_{\lambda}} (x) \circ a_i\|^2 & \leq \|a_i \| \|U_{U_{\textbf{1}-e_{\lambda}} (x)} (a_i )\| = \|a_i \| \|U_{\textbf{1}-e_{\lambda}} U_{x} U_{\textbf{1}-e_{\lambda}} (a_i )\| \\
 & \leq  \|a_i \| \|\textbf{1}-e_{\lambda}\|^2  \|x\|^2  \| U_{\textbf{1}-e_{\lambda}} (a_i)\| \leq \|a_i \| \| U_{\textbf{1}-e_{\lambda}} (a_i)\|,
 \end{aligned}$$ which combined with \eqref{eq first estimation of the norm of varphi minus varphi Ulambda} gives $$\left|\left(\varphi -  \varphi U_{e_{\lambda}} \right) (x) \right| \leq\delta^{-1} \sum_{i=1}^m \|a_i \|^{\frac12} \| U_{\textbf{1}-e_{\lambda}} (a_i)\|^{\frac12},$$ for all $x\in M(\mathfrak{A})_{sa}$ with $\|x\|\leq 1$ and $\varphi\in \mathcal{F}$. Clearly, each  $\| U_{\textbf{1}-e_{\lambda}} (a_i)\|$ tends to zero for all $i=1, \ldots, m$, and hence $\displaystyle \lim_{\lambda} \| \varphi -  \varphi U_{e_{\lambda}}\| =0$, uniformly on $\varphi \in \mathcal{F}$. 
\end{proof}

It is an open question to determine whether the conclusion in the previous proposition is actually a characterization of equicontinuity for families of J-strict continuous functionals.   

\section{An application on the extension of surjective Jordan $^*$-homomorphisms}\label{sec 4: surjective extensions} 

This final section is devoted to establish a first application of the J-strict topology in a result guaranteeing when a surjective Jordan $^*$-homomorphism between two JB$^*$-algebras admits an extension to a  surjective Jordan $^*$-homomorphism between the multipliers algebras. 
In the setting of C$^*$-algebras G.K. Pedersen was the first one observing that every surjective $^*$-homomorphism between $\sigma$-unital C$^*$-algebras extends to a surjective $^*$-homomorphism between their corresponding multipliers algebras (see \cite[Theorem 10]{PedSAW} or \cite[Proposition 3.12.10]{Ped}). The hypothesis affirming that $A$ and $B$ are $\sigma$-unital cannot be relaxed (see \cite[3.12.11]{Ped} and \cite[Proposition 6.8]{LanceBook} for further generalizations). An appropriate version of Pedersen's result for JB$^*$-algebras is the main goal of this final section.  \smallskip

We shall require some additional tools. The first one is an intermediate value type theorem for Jordan $^*$-epimorphisms between JB$^*$-algebras. The result, which generalizes \cite[Proposition 1.5.10]{Ped} and \cite[Exercise 4.6.21]{KadRingrBook1}, is interesting by itself as a potential independent tool.

\begin{theorem}\label{t intermediate value theorem for positive jordan star hom}
	Let $\Phi:\mathfrak{A}\to \mathfrak{B}$ be a Jordan $^*$-epimorphism between JB$^*$-algebras. Suppose that $0\leq b \leq d$ in $\mathfrak{B}$ and $\Phi(c)=d$ for some $c\geq 0$. Then there exists $a\in \mathfrak{A}$ such that $0\leq a\leq c$ and $b=\Phi(a)$.
\end{theorem}

\begin{proof} Up to considering the canonical unital extension of $\Phi$, we may assume that $\mathfrak{A}$ and $\mathfrak{B}$ are unital JB$^*$-algebras and $\Phi$ is unital. Since $\Phi$ is linear, surjective and symmetric, we can find $z\in \mathfrak{A}_{sa}$ such that $b=\Phi(z)$. By functional calculus, we can write $z$ uniquely in the form $z=z_{+}-z_{-}$ with $z_{+},z_{-}\geq 0$ and $z_{+} \perp z_{-} $. Since $\Phi$ is positive and preserves orthogonality we deduce from $b=\Phi (z_{+})-\Phi(z_{-})$ and $\Phi (z_{+})\perp \Phi(z_{-})$ that $\Phi(z_{-})=0$. Thus, by replacing $z$ with $z_{+},$ if necessary, we can assume that $z$ is positive in $\mathfrak{A}$.\smallskip
	
Let us write $(z-c)$ as the difference of two orthogonal positive elements $( z-c)_{+}$ and $( z-c)_{-}$ and set $x=( z-c)_{+}.$ Since $\Phi$ preserves functional calculus, we have $$\Phi(x)=(\Phi (z-c))_{+}=(b-d)_{+}=0.$$
	
Moreover, from $z-c=(z-c)_{+}-(z-c)_{-}\leq (z-c)_{+}$ we deduce that $z\leq x+c$. For each natural $n$ define $y_n:=U_{c^{\frac{1}{2}}} U_{(\frac{1}{n}\textbf{1}+x+c)^{-\frac{1}{2}}} (z)$. It is clear that $(y_n)\subseteq \mathfrak{A}_{+}$ --we note that in case we consider the unitization of $\mathfrak{A}$, because $\mathfrak{A}$ is not unital, the element $y_n$ lies in $\mathfrak{A}$ for all $n$.  We shall show that the sequence $(y_n)_n$ is convergent. Let $\mathfrak{C}$ denote the JB$^*$-subalgebra of $\mathfrak{A}$ generated by $\textbf{1}, z$ and $c$ (where $\textbf{1}=\textbf{1}_{\mathfrak{A}^{**}}$). Observe that the whole JB$^*$-subalgebra of $\mathfrak{A}$ generated by $\textbf{1}$ and the hermitian element $z-c$ (which is isometrically isomorphic to a commutative unital C$^*$-algebra) is contained in $\mathfrak{C}$, and consequently, by functional calculus, $x=( z-c)_{+}\in \mathfrak{C}$. Therefore the elements  $\textbf{1}, c, z, x, x+c,\frac{1}{n}\textbf{1}+x+c$ all lie in $\mathfrak{C}$. By the Shirshov-Cohn theorem (see \cite[2.4.14 and 2.4.15]{HOS} or \cite[Corollary 2.2]{Wright77}) $\mathfrak{C}$ is a JC$^*$-algebra, that is, $\mathfrak{C}$ is a JB$^*$-subalgebra of some $B(H)$. Thus for $u,v\in \mathfrak{C}$ its Jordan product in $\mathfrak{C}$ (and hence in $\mathfrak{A}$) is induced by the associative product of $B(H)$, that is  $u\circ v=\frac{1}{2}(u\cdot v+v\cdot u)$ where $u\cdot v$ denotes the (associative) product in $B(H)$. The involution is the same in all the structures considered here. Working in $B(H)$ we can apply an argument in \cite[Lemma 1.4.4]{Ped} or \cite[Exercise 4.6.21]{KadRingrBook1}, it is included here for completeness. We set $u_n=c^{\frac{1}{2}}\cdot  (\frac{1}{n}\textbf{1} +x+c)^{-\frac{1}{2}} \cdot z^{\frac{1}{2}} \in B(H)$ and $d_{nm}= (\frac{1}{n}\textbf{1}+x+c)^{-\frac{1}{2}}- (\frac{1}{m}\textbf{1}+x+c)^{-\frac{1}{2}}$. By observing that $c\leq x+c$ and $z\leq x+c$, the positivity of the involved $U$-operators, and that the elements $(x+c),$ $(\frac{1}{n}\textbf{1}+x+c)$, $(x+c)^{\frac12},$ $(\frac{1}{n}\textbf{1}+x+c)^{-\frac{1}{2}}$ and $d_{nm}$ all commute in $B(H)$ we have 
\begin{equation}\label{eq norm in B(H) with Pedersen method} \begin{aligned}\|u_n-u_m\|^2 &= \| c^{\frac{1}{2}} \cdot d_{nm} \cdot z^{\frac{1}{2}}\|^2 =\|c^{\frac{1}{2}}\cdot d_{nm}\cdot z\cdot d_{nm} \cdot c^{\frac{1}{2}} \| \\ &\leq \|c^{\frac{1}{2}} \cdot d_{nm} \cdot (x+c) \cdot d_{nm} \cdot c^{\frac{1}{2}} \| = \|(x+c)^{\frac12} \cdot d_{nm} \cdot c^{\frac12} \|^2 \\
&=\|(x+c)^{\frac12}\cdot d_{nm}\cdot c \cdot d_{nm} (x+c)^{\frac12} \| \\
&\leq \|(x+c)^{\frac12} \cdot d_{nm} \cdot (x+c)  \cdot d_{nm} \cdot (x+c)^{\frac12} \| \\
&=\| (x+c)^{2} \cdot d_{nm}^2\| = \| (x+c) \cdot d_{nm} \|^2 \longrightarrow_{n,m\to \infty} 0,
	\end{aligned} 
\end{equation}  where to compute the last norm we work in the abelian unital C$^*$-algebra generated by the positive element $x+c$ and the unit of $B(H)$ identified with $C(\sigma(x+c))$ in such  away that $x+c$ corresponds to the embedding of $\sigma(x+c)$ into $\mathbb{C}$. In the unital and commutative C$^*$-algebra $C([0,M])$ ($M\in \mathbb{R}^+$), the sequence $\left( \frac{t}{\sqrt{\frac1n +t}} \right)_n$ is pointwise increasing to $\sqrt{t}$, it follows from Dini's theorem that the convergence is uniform, and hence $\displaystyle \lim_{n,m\to \infty} \left\| \frac{t}{\sqrt{\frac1n +t}} -\frac{t}{\sqrt{\frac1m +t}}\right\|_{\infty}= 0,$ which proves the limit employed in the last line of \eqref{eq norm in B(H) with Pedersen method}.\smallskip

Therefore the sequence $(u_n)_n$ is norm convergent in $B(H)$ --however, it does not belong to $\mathfrak{C}$--, and thus $$\begin{aligned}
(u_n \cdot u_n^*)_n &= \left(c^{\frac{1}{2}}\cdot \left(\frac{1}{n}\textbf{1} +x+c\right)^{-\frac{1}{2}}\cdot z \cdot  \left(\frac{1}{n}\textbf{1} +x+c\right)^{-\frac{1}{2}}\cdot c^{\frac{1}{2}}\right)_n \\
	&= \left(U_{c^{\frac{1}{2}}} U_{(\frac{1}{n}\textbf{1}+x+c)^{-\frac{1}{2}}} (z) \right)_n  = (y_n)_n \subset \mathfrak{C}\subset \mathfrak{A}
\end{aligned}$$ converges in norm too. Furthermore, $$y_n = U_{c^{\frac{1}{2}}} U_{(\frac{1}{n}\textbf{1}+x+c)^{-\frac{1}{2}}} (z) \leq U_{c^{\frac{1}{2}}} U_{(\frac{1}{n}\textbf{1}+x+c)^{-\frac{1}{2}}} (x+c)\leq U_{c^{\frac{1}{2}}} (\textbf{1}) = c. $$\smallskip

Let $a\in \mathfrak{A}$ denote the limit of the sequence $(y_n)_n$ --if the original JB$^*$-algebra $\mathfrak{A}$ were not unital, the elements $y_n$ and $a$ all lie in $\mathfrak{A}$. Since $0\leq y_n\leq c$ for all $n$, it follows that $a\leq c$. Now by applying that $\Phi$ is a unital Jordan $^*$-homomorphism with $\Phi(x)=0,$ $\Phi(z)=b$ and $\Phi(c)=d$ we have 
\begin{equation}\label{eq Phi yn}
\begin{aligned}
	\Phi (y_n) & = \Phi \left( U_{c^{\frac{1}{2}}} U_{(\frac{1}{n}\textbf{1}+x+c)^{-\frac{1}{2}}} (z) \right) = U_{\Phi(c)^{\frac{1}{2}}} U_{(\frac{1}{n}\textbf{1}+\Phi(x) + \Phi(c))^{-\frac{1}{2}}} (\Phi(z)) \\
	&= U_{d^{\frac{1}{2}}} U_{(\frac{1}{n}\textbf{1}+  d)^{-\frac{1}{2}}} (b) = U_{d^{\frac{1}{2}}\circ (\frac{1}{n}\textbf{1}+  d)^{-\frac{1}{2}}} (b),
\end{aligned} 
\end{equation}   where in the last equality we applied that $d$ and $\frac{1}{n}\textbf{1}+  d$ (and hence $d^{\frac{1}{2}}$ and $(\frac{1}{n}\textbf{1}+  d)^{-\frac{1}{2}}$) operator commute. It is well known that the sequence $\left(d^{\frac{1}{2}}\circ (\frac{1}{n}\textbf{1}+  d)^{-\frac{1}{2}}\right)_n = \left(\left(d\circ (\frac{1}{n}\textbf{1}+  d)^{-1}\right)^{\frac12}\right)_n$ converges in the weak$^*$ topology and in the strong$^*$ topology of $\mathfrak{B}^{**}$ to the range projection $r(d)$ of $d$. Since the product of $\mathfrak{B}^{**}$ is jointly strong$^*$-continuous on bounded sets, we deduce that $\left(  U_{d^{\frac{1}{2}}\circ (\frac{1}{n}\textbf{1}+  d)^{-\frac{1}{2}}} (b) \right)_n \to U_{r(d)} (b)$ in the strong$^*$ topology of $\mathfrak{B}^{**}.$ Moreover, since $b\leq d\leq r(d)$ we also have $U_{r(d)} (b) =b$. On the other hand, since $(\Phi(y_n))_n\to \Phi(a)$ and the norm topology is stronger than the strong$^*$ topology, we deduce from \eqref{eq Phi yn} and the above conclusions that $\Phi(a) = b$.  
\end{proof}

As in the case of C$^*$-algebras, a JB$^*$-algebra is called \emph{$\sigma$-unital} if it admits a countable approximate unit. A standard argument \cite[Proposition 3.10.5]{Ped}, also valid for JB$^*$-algebras, shows that a JB$^*$-algebra $\mathfrak{A}$ is $\sigma$-unital if and 
only if it contains a strictly positive element $h\in \mathfrak{A}$ (i.e. $\phi (h) >0$ for every non-zero positive functional $ \phi\in \mathfrak{A}^*,$ or equivalently, the range projection of $h$ in $\mathfrak{A}^{**}$ is the unit).\smallskip

The next lemma is a Jordan version of \cite[Lemma 6.1]{LanceBook} and proves a new characterization of strictly positive elements. 

\begin{lemma}\label{l the U map of a strictly positive element has norm dense image} Let $h$ be a  positive element in a JB$^*$-algebra $\mathfrak{A}$, and let $\mathfrak{A}^+$ denote the set of all positive elements in $\mathfrak{A}$. Then the following statements are equivalent: \begin{enumerate}[$(a)$]
		\item $h$ is strictly positive.
		\item $U_{h} (\mathfrak{A})$ is norm dense in $\mathfrak{A}$, that is, the inner ideal generated by $h$ is the whole $\mathfrak{A}$.
		\item $U_{h} (\mathfrak{A}_{sa})$ is norm dense in $\mathfrak{A}_{sa}.$
		\item $U_{h} (\mathfrak{A}^+)$ is norm dense in $\mathfrak{A}^+.$ 
	\end{enumerate} 
\end{lemma}

\begin{proof} The identity $U_{U_h(a)} (b) = U_h U_a U_h (b)$ shows that $\mathfrak{A}(h) = \overline{U_h (\mathfrak{A})}$ is a norm-closed quadratic ideal of $\mathfrak{A}$. \smallskip
	
$(a)\Rightarrow (b)$ If  $h$ is strictly positive, every positive functional in $\mathfrak{A}^*$ vanishing on $\mathfrak{A}(h)$ also vanishes on $\mathfrak{A}$. We claim that $r_{{\mathfrak{A}^{**}}}(h) = \textbf{1}$. Otherwise, $\textbf{1}- r_{{\mathfrak{A}^{**}}}(h)$ is a non-zero projection, and hence there exists a positive norm-one functional $\phi \in \mathfrak{A}^*$ such that $\phi (\textbf{1}- r_{{\mathfrak{A}^{**}}}(h)) =1$ (cf. \cite[Theorem or Proposition]{Bun01}). In particular, $0\leq \phi (h) \leq \phi (r_{{\mathfrak{A}^{**}}}(h)) =0$, contradicting that $h$ is strictly positive. Finally, by \cite[Proposition 2.1]{BunChuZal2000} we have  $\mathfrak{A}(h) = \mathfrak{A}^{**}_2 (r_{{\mathfrak{A}^{**}}}(h)) \cap \mathfrak{A} = \mathfrak{A}^{**}_2  (\textbf{1}) \cap \mathfrak{A} = \mathfrak{A}.$\smallskip
	
$(d)\Rightarrow (a)$ Suppose now that $\mathfrak{A}(h) =  \mathfrak{A}.$ In this case, for each positive functional $\phi\in \mathfrak{A}^*$ with $\phi(h) =0,$ and each positive element $a\in \mathfrak{A}$ we have $0\leq \phi U_h (a) \leq \|a\| \phi (h^2) =0$, and hence $\phi U_h (\mathfrak{A}) =\{0\}$, witnessing that $\phi =0$.\smallskip

The remaining implications are clear from the identity $\mathfrak{A}= \mathfrak{A}_{sa}\oplus i \mathfrak{A}_{sa}$, the fact that $U_h$ is a positive mapping, and the decomposition of every positive element as the difference of two positive elements. 
\end{proof}

By \cite[Lemma 1.33$(i)$]{AlfsenShultz2003} if $J$ is a Jordan ideal of a JB$^*$-algebra $\mathfrak{A}$ and $(u_{\lambda})$ is an approximate identity of $J$ we have $$\|a+J_{sa}\|=\lim_{\lambda} \|a-u_{\lambda}\|=\lim_{\lambda}\|U_{1-u_{\lambda}}(a)\|,$$ for every $a\in \mathfrak{A}_{sa}$. Given $a\in \mathfrak{A}_{sa},$ as shown in the proof of \cite[Theorem 3.2]{Wright77} $\|a+J_{sa}\|=\|a+J\|$ whence $$\|a+J\|=\displaystyle  \lim_{\lambda} \|a-u_{\lambda}\|=\lim_{\lambda}\|U_{1-u_{\lambda}}(a)\|.$$
 The next technical lemma proves a similar identity for elements of the form $\{a,b,a\}$ with $a,b\in \mathfrak{A}_{sa}$ and $b\geq 0$.

\begin{lemma}\label{l norm element ideal}
Let $J$ be a Jordan $^*$-ideal of a JB$^*$-algebra $\mathfrak{A}.$ Let $(u_{\lambda})$ be an approximate unit for $J$. Then the equality
	$$ \|U_a (b) +J\|=\lim_{\lambda} \left\| U_a U_{b^{\frac12}}(\textbf{1}-u_{\lambda})\right\|$$ holds for all $a,b\in \mathfrak{A}_{sa}$ with $b$ positive.	
\end{lemma}

\begin{proof} Up to replacing $\mathfrak{A}$ with its unitization we can always assume that $\mathfrak{A}$ is unital.  As we have seen in the comments above, it is enough to show that $\displaystyle \|U_a (b)+J_{sa}\|=\lim_{\lambda}\| U_a U_{b^{\frac12}}(\textbf{1}-u_{\lambda}) \|$. By \cite[Lemma 1.33]{AlfsenShultz2003} we have 
	$$\|a+ J_{sa}\|=\lim_{\lambda} \| a-a\circ u_{\lambda}\|=\lim_{\lambda} \|U_{1-u_{\lambda}} (a) \|$$ for all $a\in \mathfrak{A}_{sa}$. By the study of Jordan ideals of JB-algebras developed by M. Edwards in \cite{Edw77}, there exists a projection $p$ in $\mathfrak{A}^{**}$ such that  $(u_{\lambda})\to p$ in the strong$^*$ topology of $\mathfrak{A}^{**}$ --$p$ is actually the unit of $J^{**}$. Moreover, $p$ is a central projection in $\mathfrak{A}^{**}$ and
	$$J_{sa}=U_p(\mathfrak{A}_{sa}^{**})\cap \mathfrak{A}_{sa}, \; J_{sa}^{\circ \circ}=\overline{J_{sa}}^{w^*}=U_p(\mathfrak{A}_{sa}^{**})$$ (cf. \cite[Theorems 2.3 and 3.3]{Edw77}).  We also have $a-a\circ u_{\lambda}\to a-a\circ p$ and $U_{1-u_{\lambda}} (a)\to U_{1-p}(a)$ in the strong$^*$ topology of $\mathfrak{A}^{**}$.\smallskip
	
	Since the operators $U_a$ and $U_{b^{\frac{1}{2}}}$ are both positive, the net $ \left(U_a U_{b^{\frac12}}(\textbf{1}-u_{\lambda})\right)$ is decreasing and converges strongly to $ U_a U_{b^{\frac12}}(\textbf{1}-p).$ By Dini's theorem, applied to the previous elements regarded as continuous functions on the space of all positive functionals in the closed unit ball of $\mathfrak{A}^*$, we have $$\lim_{\lambda} \left\| U_a U_{b^{\frac12}}(\textbf{1}-u_{\lambda}) - U_a U_{b^{\frac12}}(\textbf{1}-p) \right\|=0.$$ It remains to show that
	$$\|U_a (b) +J_{sa}\|=\| U_a U_{b^{\frac12}}(\textbf{1}-p)\|. $$ Namely, $\big( \mathfrak{A}_{sa}/J_{sa} \big)^{**}=\mathfrak{A}_{sa}^{**}/J_{sa}^{\circ \circ}=\mathfrak{A}_{sa}^{**}/U_p(\mathfrak{A}_{sa}^{**})=U_{1-p}(\mathfrak{A}_{sa}^{**})$. Let $\pi: \mathfrak{A}_{sa}\to \mathfrak{A}_{sa}/J_{sa}$ denote the canonical projection of $\mathfrak{A}$ onto $\mathfrak{A}_{sa}/J_{sa}$. Then $\pi^{**}:\mathfrak{A}_{sa}^{**}\to (\mathfrak{A}/J)_{sa}^{**}$ is given by $\pi^{**}(a)=(1-p)\circ a, \; a\in \mathfrak{A}_{sa}^{**}$.\smallskip
	
	Since $p$ is central (and hence $1-p$ also is central), it can be easily checked that  $U_{b^{\frac12}} (\textbf{1}-p) =(1-p)\circ b$ and $$\begin{aligned} U_a ((1-p)\circ b) & =2a\circ (a\circ ( (1-p)\circ b))-a^2\circ ((1-p)\circ b) \\
		&=2a\circ ((1-p)\circ ( a\circ b))-(1-p)\circ (a^2\circ b) \\
		&=(1-p)\circ (2a \circ ( a\circ b))-(1-p)\circ (a^2\circ b)=(1-p)\circ U_a (b).
	\end{aligned}$$ Therefore
	$$\|U_a (b) +J_{sa}\|=\|\pi^{**}(U_a (b))\|=\|(1-p)\circ U_a (b)\|=\| U_a U_{b^{\frac12}}(\textbf{1}-p)\|. $$
\end{proof}

We can state now the desired extension of Pedersen's theorem commented at the beginning of this section.  

\begin{theorem}\label{t extension of onto Jordan starhom sigma unital}  Let $\mathfrak{A}$ and $\mathfrak{B}$ be two JB$^*$-algebras and assume that $\mathfrak{A}$ is $\sigma$-unital. Suppose $\Phi : \mathfrak{A}\to \mathfrak{B}$ is a surjective Jordan $^*$-homomorphism. Then there exists an extension of $\Phi$ to a surjective J-strict continuous Jordan $^*$-homomorphism $\tilde{\Phi} :  M(\mathfrak{A})\to M(\mathfrak{B})$.
\end{theorem}

\begin{proof} The existence of a J-strict continuous  Jordan $^*$-homomorphism $\tilde{\Phi} :  M(\mathfrak{A})\to M(\mathfrak{B})$ extending $\Phi$ is guaranteed by Proposition~\ref{p extensions of surjective Jordan homomorphisms}. We need to prove that $\tilde{\Phi}$ is surjective. It suffices to show that every positive element $z\in M(\mathfrak{B})$ is in the image of $\tilde{\Phi}$.  \smallskip
	
%We observe that $\mathfrak{B}$ is $\sigma$-unital too. Let $h$ be a strictly positive element in $\mathfrak{A}$, then $\Phi (h)$ is a positive element in $\mathfrak{B}$. If $\varphi$ is a positive functional in $\mathfrak{B}^*$, the mapping $\varphi\Phi: \mathfrak{A}\to \mathbb{C}$ also is a positive functional on $\mathfrak{A}$. If $\varphi \Phi (h) =0$, the assumption on $h$ proves that $\varphi \Phi =0$, and hence $\varphi=0$ by the surjectivity of $\Phi$. Therefore, $ \Phi (h)$ is strictly positive in $\mathfrak{B}$.\smallskip
	
Let $z$ be a positive element in $M(\mathfrak{B})$. By hypothesis we can find a countable approximate unit $(u_n)_n$ of $\mathfrak{A}$. Clearly $(u_n)_n\to \textbf{1}_{_{\mathfrak{A}}}$ in the J-strict topology of $M(\mathfrak{A})$, and thus $(\Phi(u_n))_n\to \tilde{\Phi}(\textbf{1}_{_{\mathfrak{A}}})$ in the J-strict topology of $M(\mathfrak{B})$. On the other hand, since $\Phi$ is surjective, for each $b = \Phi (a)\in \mathfrak{B}$ (with $a\in \mathfrak{A}$), we have $$b \circ (\Phi(u_n)- \textbf{1}_{_{\mathfrak{B}}}) = \Phi(a) \circ (\Phi(u_n)- \textbf{1}_{_{\mathfrak{B}}}) =  \Phi(a \circ u_n)- b\to b-b=0,$$ in norm by the continuity of $\Phi$ and the fact that $(u_n)$ is an approximate unit for $\mathfrak{A}$. We deduce that $(\Phi(u_n))_n\to \textbf{1}_{_{\mathfrak{B}}}$ in the J-strict topology of $M(\mathfrak{B})$, and since this topology is Hausdorff we get $\tilde{\Phi}( \textbf{1}_{_{\mathfrak{A}}}) = \textbf{1}_{_{\mathfrak{B}}}.$\smallskip

We define a sequence $(b_n)_n\subseteq \mathfrak{B}$ given by $$b_n := U_{z^{\frac12}} (\Phi (u_n)) = 2 (z^{\frac12}\circ \Phi (u_n))\circ z^{\frac12} - z\circ \Phi (u_n).$$  Since $(\Phi(u_n))_n\to \textbf{1}_{_{\mathfrak{B}}}$ in the J-strict topology of $M(\mathfrak{B})$, by Remark~\ref{r Lemma 3.3 is valid for striclty convergent nets of sa elements with limit 1}, the sequences $( z\circ \Phi (u_n))_n$ and  $(z^{\frac12}\circ \Phi (u_n))_n$ converge in the J-strict topology to $z$ and $z^{\frac12}$, respectively. A new application of  Remark~\ref{r Lemma 3.3 is valid for striclty convergent nets of sa elements with limit 1} proves that $((z^{\frac12}\circ \Phi (u_n))\circ z^{\frac12} )_n\to z$ in the J-strict topology. We have therefore proved that $(b_n)_n\to z$ in the J-strict topology. We further know that since $(\Phi (u_n))_n$ is monotone increasing and $U_{z^{\frac12}}$ is a positive operator, the sequence $(b_n)_n$ is monotone increasing too.\smallskip

Let $h$ be a norm-one strictly positive element in $\mathfrak{A}$ whose existence is guaranteed by the hypotheses.  It follows from the above that the sequences $(b_n\circ \Phi(h))_n$ and  $(b_n\circ \Phi(h^2))_n$ are norm Cauchy. Up to considering an appropriate subsequence, we may assume that \begin{equation}\label{eq bn+1 - bn times Phi(h)}
\left\| (b_{n+1}-b_n) \circ \Phi(h) \right\|, \ \left\| (b_{n+1}-b_n) \circ \Phi(h^2) \right\| < 4^{-n}, \hbox{ for all natural } n.
\end{equation}

We shall prove by induction the existence of a sequence of positive elements $(a_n)_n\subset M(\mathfrak{A})$ satisfying \begin{equation}\hbox{ $a_n\leq a_{n+1}\leq 1,$ $\Phi (a_n) = b_n,$ and $\| U_h (a_{n+1}-a_n) \| < 2^{-n}$ for all $n\in \mathbb{N}$.}\end{equation} We shall work with the unital extensions of $\Phi$, $\mathfrak{A}$ and $\mathfrak{B}$, the extension of $\Phi$ coincides with an appropriate restriction of $\tilde{\Phi}$. We shall denote the unital extension of $\Phi$ by the same symbol $\Phi$. Since the unital extension of $\Phi$ is onto, Theorem~\ref{t intermediate value theorem for positive jordan star hom} proves the existence of a positive $a_1\in \mathfrak{A}\oplus \mathbb{C} \textbf{1}$ such that $ \tilde{\Phi} (a_1 )= b_1$ and $0\leq a_1\leq \textbf{1}$. Suppose we have already defined the elements $a_1\leq a_2\leq \ldots\leq a_n$ in $M(\mathfrak{A})$ satisfying the above properties. Since $\tilde{\Phi}(\textbf{1}-a_n) = \textbf{1}-b_n \geq b_{n+1}-b_n$, Theorem~\ref{t intermediate value theorem for positive jordan star hom}, applied to the unital extension of $\Phi$, assures the existence of $0\leq c\leq \textbf{1}-a_n$ in $M(\mathfrak{A})$ such that $\Phi (c) = b_{n+1} -b_n$. The element we need will follow after a perturbation by an element in the kernel of $\Phi$. The kernel of $\Phi$, $\ker(\Phi)$, is a Jordan $^*$-ideal of $\mathfrak{A}_1=\mathfrak{A}\oplus \mathbb{C}\textbf{1}$ and the quotient mapping $[\Phi]: \mathfrak{A}_1/\ker(\Phi) \to  \mathfrak{B}_1=\mathfrak{B}\oplus \mathbb{C}\textbf{1}$ is a Jordan $^*$-isomorphism, in particular an isometry. Let $(v_j)_j$ be an approximate identity of $\ker(\Phi)$. By applying  Lemma~\ref{l norm element ideal}, the previous property of $[\Phi]$, and \eqref{eq bn+1 - bn times Phi(h)} we get $$
\begin{aligned}\lim_{j} \left\| U_{h} U_{c^{\frac12}} (\textbf{1}-v_j)  \right\| &= \left\| U_{h} ( c) +\ker (\Phi) \right\|  = \left\| [\Phi] \left(U_{h} ( c) +\ker (\Phi)\right) \right\| \\
&= \left\| U_{\Phi(h)} ( \Phi(c))  \right\| = \left\| U_{\Phi(h)} ( b_{n+1}-b_{n})  \right\| \\
&\leq 2 \left\| \left(\Phi(h)\circ ( b_{n+1}-b_{n})\right) \circ \Phi(h)  \right\| + \left\| \Phi(h)^2 \circ ( b_{n+1}-b_{n})  \right\| \\
&< 2 \ 4^{-n} + 4^{-n} \leq  2^{-n} \ \ \ \hbox{ for all } n>2.
\end{aligned}$$ We can therefore find $j_0$ such that  $\left\| U_{h} U_{c^{\frac12}} (\textbf{1}-v_{j_0})  \right\|< 2^{-n}$. We set $$a_{n+1} := a_n + U_{c^{\frac12}} (\textbf{1}-v_{j_0}).$$ Clearly $$ a_{n} \leq a_{n+1} \leq a_n + U_{c^{\frac12}} (\textbf{1}) = a_n + c\leq a_n + \textbf{1}-a_n = \textbf{1}.$$ By definition $$\Phi (a_{n+1}) = \Phi(a_n) + \Phi (c) - \Phi\left( U_{c^{\frac12}} (v_{j_0}) \right) = b_n + b_{n+1} -b_n = b_{n+1},$$ where we employed that $U_{c^{\frac12}} (v_{j_0})\in \ker(\Phi)$ because $v_{j_0}\in\ker(\Phi)$. Furthermore, by construction $$\left\| U_h (a_{n+1}- a_{n}) \right\| = \left\| U_h U_{c^{\frac12}} (\textbf{1}-v_{j_0}) \right\|< 2^{-n}.$$    

Let us now proceed with the final step in our argument. Fix an arbitrary positive $d\in \mathfrak{A}$. By combining the properties of the elements $a_n$'s, the inequality in \cite[Lemma 3.5.2$(ii)$]{HOS} (see \eqref{inequelity in HOS}), and the fundamental identity \eqref{eq main identity U maps} we derive that  
$$\begin{aligned}
\left\| U_h (d) \circ (a_{n+1}-a_n) \right\|^2 &\leq \| a_{n+1}-a_n\| \ \left\| U_{U_h (d)} (a_{n+1}-a_n) \right\| \\
&= \| a_{n+1}-a_n\| \ \left\| U_h U_d U_h (a_{n+1}-a_n) \right\| \\
&\leq 2 \|d\|^2 \ \left\| U_h (a_{n+1}- a_{n}) \right\| < \|d\| \ 2^{-n+1}.
\end{aligned} $$ The arbitrariness of $d$ together together with the fact that $h$ is strictly positive and Lemma~\ref{l the U map of a strictly positive element has norm dense image} prove that $(a_n)_n$ is a J-strict Cauchy sequence in $M(\mathfrak{A})$. Theorem~\ref{t the strict topology is complete} implies that $(a_n)_n$ converges to some $a_0$ in the J-strict topology of $M(\mathfrak{A})$. Since $(\tilde{\Phi} (a_n))_n = (b_n)_n\to z$ in the J-strict topology and $ \tilde{\Phi}$ is J-strict continuous, we obtain $\tilde{\Phi} (a_0) =z$, which finishes the proof.
\end{proof}

\begin{proposition}\label{p extension of onto triple hom sigma unital}  Let $\mathfrak{A}$ and $\mathfrak{B}$ be two JB$^*$-algebras and assume that $\mathfrak{A}$ is $\sigma$-unital. Suppose $\Phi : \mathfrak{A}\to \mathfrak{B}$ is a surjective triple homomorphism. Then there exists an extension of $\Phi$ to a surjective J-strict continuous triple homomorphism $\tilde{\Phi} :  M(\mathfrak{A})\to M(\mathfrak{B})$.
\end{proposition}

\begin{proof} As in the proof of Proposition~\ref{p extensions of surjective triple homomorphisms}, $\Phi^{**}:\mathfrak{A}^{**} \to \mathfrak{B}^{**}$ is a weak$^*$ continuous triple homomorphism, $\Phi^{**} (M(\mathfrak{A})) \subseteq M(\mathfrak{B})$, and the element $u=\Phi^{**} (\textbf{1})$ is a unitary tripotent in $M(\mathfrak{B})$. Furthermore, keeping the notation in Remark~\ref{r u-isotopes}, $\mathfrak{B} (u)$ is a JB$^*$-subalgebra of $M(\mathfrak{B})_2 (u)$ and the mapping $\Phi : \mathfrak{A}\to \mathfrak{B} (u)$ is a surjective and unital Jordan $^*$-homomorphism. We are in a position to apply Theorem~\ref{t extension of onto Jordan starhom sigma unital} to $\Phi : \mathfrak{A}\to \mathfrak{B} (u)$, to find an extension of $\Phi$ to a surjective and J-strict continuous Jordan $^*$-homomorphism $\tilde{\Phi} :  M(\mathfrak{A})\to M(\mathfrak{B}(u)) = M(\mathfrak{B})$. Clearly, $\tilde{\Phi}$ is a triple homomorphism by the uniqueness of the triple product.
\end{proof}

When in the proof of Corollary~\ref{c extension of OP }, Proposition~\ref{p extensions of surjective triple homomorphisms} is replaced with Proposition~\ref{p extension of onto triple hom sigma unital} we get the next result.

\begin{corollary}\label{c extension of OP surjective with strictly positive element}  Let $\mathfrak{A}$ and $\mathfrak{B}$ be two JB$^*$-algebras and assume that $\mathfrak{A}$ is $\sigma$-unital. Then every continuous and surjective orthogonality preserving operator $\Phi:\mathfrak{A} \to \mathfrak{B}$ extends to a surjective J-strict continuous orthogonality preserving operator $\tilde{\Phi}: M(\mathfrak{A)}\to M(\mathfrak{B})$.
\end{corollary}

\section{Open problems}

The reader is already aware of the close parallelism between the (C$^*$-)strict and strong$^*$ topologies in a C$^*$-algebra $A$. More concretely, as we noted in Remark~\ref{r final comments with seminorms related to Hilbertian seminorms}, the C$^*$-strict topology of $M(A)$ coincides with the topology generated by the seminorms of the form $x\mapsto   \| a (x \circ x^*) a\|^{\frac12}$ ($x\in M(A)$) with $a$ in $A_{sa},$ while the C$^*$-strong$^*$ topology of $A^{**}$ is the one determined by all the preHilbertian seminorms defined by  $\displaystyle x\mapsto  \varphi\left( x \circ x^* \right)^{\frac12}$ ($x\in M(A)$) with $\varphi$ running in the set of all normal states on $A^{**}$. The strong$^*$ topology of $A$ is the restriction of the corresponding topology of $A^{**}$ to $A$. The celebrated Sakai's theorem proves that the product of each von Neumann algebra $W$ is separately weak$^*$ continuous, and a fundamental property of the strong$^*$ topology asserts that the product of $W$ is jointly strong$^*$ continuous on bounded sets (cf. \cite[Proposition 1.8.12]{Sa} and \cite{RodPa91} for a generalization in the case of JBW$^*$-algebras and JBW$^*$-triples).\smallskip

It is not hard to see that the strict topology of the multipliers algebra of a C$^*$-algebra $A$ enjoys similar properties to those described for the weak$^*$  and the strong$^*$ topologies of a von Neumann algebra. Namely, suppose $(y_\lambda)_{\lambda}\to y$ with respect to the strict topology of $M(A)$. Given $x\in M(A)$ and $a\in A$,  we have $\| a (x y_\lambda - xy)  \| = \| (a x) ( y_\lambda - y)  \| \to 0,$  since $ ax \in A$, and $\| (x y_\lambda - x y )a   \| = \| x ((  y_\lambda - y) a)  \| \to 0,$ by the assumptions on   $(y_\lambda)_{\lambda}$. Therefore,  $(x y_\lambda)_{\lambda}\to x y$ in the strict topology of $M(A)$. Similarly, $( y_\lambda x)_{\lambda}\to y x$ with respect to the strict topology of $M(A)$. This shows that the product of $M(A)$ is separately strict continuous. \smallskip

If we take now two bounded nets $(y_\lambda)_{\lambda}\to y$ and $(x_\mu)_{\mu}\to x$ with respect to the strict topology of $M(A)$ and an element $a\in A$, the inequalities  
$$\begin{aligned}
\| a ( y_\lambda x_\mu - y x )\| & \leq \| a ( y_\lambda x_\mu - y x_{\mu} )\| + \| a ( y x_\mu - y x )\| \\
&\leq \| a ( y_\lambda  - y )\| \|x_\mu\| + \| (a  y) (x_\mu - x )\|\end{aligned}$$
$$\begin{aligned}
	\| ( y_\lambda x_\mu - y x ) a \| & \leq \| ( y_\lambda x_\mu - y_{\lambda} x ) a\| + \| ( y_\lambda x - y x ) a \| \\
	&\leq \|y_\lambda\| \| (x_\mu -x ) a\| + \| ( y_\lambda - y ) x a \|,\end{aligned}$$ assure that $( a ( y_\lambda x_\mu - y x ))$ and $( ( y_\lambda x_\mu - y x )a )$ are norm null nets, and hence $(x_{\mu} y_\lambda)_{\lambda}\to x y$ with respect to the strict topology of $M(A)$. That is, the product of $M(A)$ is jointly strict continuous on bounded sets.\smallskip 

Let $\mathfrak{A}$ be a JB$^*$-algebra. The following two natural questions remain open: 
\begin{enumerate}[$(Q1)$]\item Is the Jordan product of $M(\mathfrak{A})$ separately J-strict continuous (on bounded sets)?
\item Is the Jordan product of $M(\mathfrak{A})$ jointly J-strict continuous on bounded sets?
\end{enumerate}

\smallskip\smallskip\smallskip

\noindent\textbf{Funding} First, second, and fourth  authors partially supported by grant PID2021-122126NB-C31 funded by MCIN/AEI/10.13039/501100011033 and by ``ERDF A way of making Europe'', Junta de Andaluc\'{\i}a grants FQM375 and PY20$\underline{\ }$00255, and by the IMAG--Mar{\'i}a de Maeztu grant CEX2020-001105-M/AEI/10.13039/ 501100011033. Third author partially supported by NSF of China (12171251).

\smallskip\smallskip

\noindent\textbf{Data Availability} Statement Data sharing is not applicable to this article as no datasets were generated or analysed during the preparation of the paper.\smallskip\smallskip

\noindent\textbf{Declarations}
\smallskip\smallskip

\noindent\textbf{Conflict of interest} The authors declare that they have no conflict of interest.

\end{document}